%% file: SIAM.tex
\documentclass[review,onefignum,onetabnum]{siamart171218}

\usepackage{amsmath,amsfonts,amssymb,amscd,subeqnarray}
\usepackage{graphicx,subfigure}
\usepackage{bm,color}
\usepackage{comment}
\usepackage{stmaryrd}
\usepackage{listings}
\usepackage{boldfonts}
\usepackage{symbols}
\usepackage{siunitx}
\usepackage{booktabs}

\input{ex_shared}

\ifpdf
\hypersetup{
  pdftitle={Positivity and maximum principle preserving discontinuous Galerkin finite element schemes for a coupled flow and transport},
  pdfauthor={D. Doe, P. T. Frank, and J. E. Smith}
}
\fi

\myexternaldocument{ex_supplement}

\begin{document}

\maketitle

\begin{abstract}
We introduce a new concept of the locally conservative flux and investigate its relationship with the compatible discretization pioneered by Dawson, Sun and Wheeler \cite{DawsonC_SunS_WheelerM-2004aa}. 
We then demonstrate how the new concept of the locally conservative flux can play a crucial role in obtaining the $L^2$ norm stability of the discontinuous Galerkin finite element scheme for the transport in the coupled system with flow. In particular, the lowest order discontinuous Galerkin finite element for the transport is shown to inherit the positivity and maximum principle when the locally conservative flux is used, which has been elusive for many years in literature. The theoretical results established in this paper are based on the equivalence between Lesaint-Raviart discontinuous Galerkin scheme and Brezzi-Marini-S{\"u}li discontinuous Galerkin scheme for the linear hyperbolic system as well as  the relationship between the Lesaint-Raviart discontinuous Galerkin scheme and the characteristic method along the streamline. Sample numerical experiments have also been performed to justify our theoretical findings. 
\end{abstract}



\section{Introduction}

The conservation is known to be crucial to preserve in simulating the coupled flow and transports \cite{john2017divergence}. One important example of the coupled flow and transports is the non-Newtonian flow models  \cite{Lee2009,leethesis2004,kang2016three}. There are a number of numerical schemes dedicated to preserve the conservation in literature, which include the well-known mixed finite element methods \cite{ewing1984convergence}, discontinous Galerkin finite element methods \cite{riviere2008discontinuous} and other relevant methods \cite{sun2009locally,LeeLeeWhi15,choi;jo;kwak;lee}. The issue of mass conservation in numerical methods for flow coupled to transport can be found at 
\cite{localconservation2002,localconservation1998,DawsonC_SunS_WheelerM-2004aa, sun2009locally} and references cited therein. On the other hand, to the best knowledge of authors, only a few discussion can be found in literature that show rigorously why the locally conservative flux is important \cite{john2017divergence,DawsonC_SunS_WheelerM-2004aa}. The most rigorous discussion is believed to be at the celebrating work by Dawson, Sun and Wheeler \cite{DawsonC_SunS_WheelerM-2004aa}. It discusses the importance of the conservative flux for computing solutions to transports using the notion of compatibility condition imposed on the flow equation. The compatibility condition is a necessary condition for the flux to guarantee that the discrete system for the transports possesses the so-called zeroth-order accuracy. The zeroth-order accuracy means that the constant true solution can be recovered. We observe that the compatibility condition is basically the local conservation, which is a bit stronger than the standard local conservation. The standard local conservation refers to conservation of mass over a control volume or an element in a finite element or finite difference grid \cite{CockburnB_DawsonC-2002aa, Hughes2000467, LeeLeeWhi15}. 

In this paper, our concern is solely restricted to the discontinuous Galerkin finite element discretizations for both flow (pressure) and transports. For the sake of presentation, we shall denote the degree of polynomials used to approximate the pressure by $k_p$ and the degree of polynomials used to approximate the solution to the transports by $k_c$, respectively. We first introduce a new concept of the local conservation, and then investigate its relationship with the compatibility condition. More precisely, let us assume that the exact flux satisfies the continuity equation 
\begin{equation}
\nabla \cdot \bu = f, \quad \mbox{ in } \Omega.
\end{equation} 
Then for a given triangulation $\mathcal{T}_h$ and $k \geq 0$, we say that the numerical flux $\bU$ is \textit{locally conservative of degree $k$} if and only if it satisfies the following equation: for all $T \in \mathcal{T}_h$, 
\begin{equation}\label{cons}
\int_T (\nabla \cdot \bU ) w \, dx = \int_T f w \, dx, \quad \forall w \in \mathbb{P}_k(T),  
\end{equation}
where $\mathbb{P}_k(T)$ is the polynomials on $T$ of degree at most $k$. We remark that the standard local conservative flux means that it is locally conservative of degree $k = 0$. Of course, if $\bU$ is strongly conservative, then it is locally conservative of any degree. Namely, the newly proposed local conservation is a more general version of the standard local conservation \cite{LeeLeeWhi15}, but it is still weaker than the strong conservation. 

The first main result in the paper is that if the numerical flux $\bU$ satisfies the local conservation of degree $k \geq 2k_c$, then the solution to the transports satisfies the $L^2$ norm stability. Secondly, if the flux $\bU$ is locally conservative of degree $k \geq k_c$, which is exactly the compatibility condition, introduced in \cite{DawsonC_SunS_WheelerM-2004aa} then the zeroth-order accuracy can be obtained. Note that the locally conservative flux of degree $k \geq k_c > 0$ can be obtained only by using IIPG (Incomplete Interior Penalty Galerkin methods) \cite{wheeler1978elliptic}, which was discovered in \cite{DawsonC_SunS_WheelerM-2004aa}. SIPG (Symmetric Interior Penalty Galerkin methods) and NIPG (Nonsymmetric Interior Penalty Galerkin methods) \cite{shu_wheeler_2005} can also lead to the locally conservative flux of degree $k$ for $k = k_c = 0$. Finally, for $k_c = 0$, if the flux is locally conservative of degree $k \geq k_c$, then the discrete solution obtained from the discontinuous Galerkin finite element scheme for the transports inherits the positivity and maximum principle preserving properties of the physical solution. We note that the locally conservative flux of degree $k = 0$ can also be obtained by using the enriched Galerkin finite elements introduced in \cite{sun2009locally,LeeLeeWhi15}. We remark that the numerical results are traced back to that of Dawson, Sun and Wheeler \cite{DawsonC_SunS_WheelerM-2004aa}. Thus, the proof of the positivity and maximum principle has been elusive for many years in literature. Our technical results are based on important relationships between three schemes: the Lesaint-Raviart discontinuous Galerkin finite element method \cite{LR1974}, the Brezzi-Marini-S{\"u}li jump stablized discontinuous Galerkin finite element method \cite{brezzi2004discontinuous}
and the characteristic method along the streamline discussed in \cite{baranger1999natural}. We identified a relationship between Brezzi-Marini-S{\"u}li jump stabilized discontinuous Galerkin finite element, which is an independently interesting result. We remark that there are tremendous efforts to design positivity and maximum principle preserving scheme in literature, which include \cite{Chu2010JCP, Chu2012JSC,Chu2011review,Chu2022quadrature}. However, these are different schemes based on limiters. The main result of the paper strengthens the result by Dawson, Sun and Wheeler \cite{DawsonC_SunS_WheelerM-2004aa} and it is summarized in Table.\ref{tab:results}.

\begin{table}[h]
    \centering
    \caption{Summary of Main Results.} 
   \label{tab:results}
    \begin{tabular}{|c|c|} \hline 
         Local Conservation of Degree $k$&  Transport's solution \\ \hline 
        $k \geq 2k_c \geq 0$ & $L^2$- Stability\\ \hline 
        $k \geq k_c\geq0$ & Zeorth-Order Accuracy\\ \hline
        $k \geq 2k_c = 0$ & Positivity and Maximum Principle\\ \hline
    \end{tabular} 
    \end{table}

Throughout the paper, we use the standard notation for the Sobolev spaces such as $L^2(\Omega)$, which is the space of square integrable functions on $\Omega$, $H^m(\Omega)$ and $W^{m,p}(\Omega)$ with $1 \leq m$ and $1 \leq p \leq \infty$. We shall denote the discontinuous Galerkin finite element scheme simply by DG. In some cases, we shall specify the degree of polynomials used to construct DG. Namely, ${\rm DG_\ell}$ refers to the discontinuous Galerkin method with polynomials of degree $\ell$.   

The rest of the paper is organized as follows. In \S \ref{sec:intro}, we introduce the system of coupled flow and transports and discuss the continuous maximum and/or minimum principle. In particular, we point out how the conservative property of the flux is correlated with the establishment of the maximum principle. In \S \ref{dtransportflow}, we introduce the discontinuous Galerkin finite element schemes for the coupled flow and transports. A new concept of the local conservation of the flux is introduced and compared with the compatibility condition. In \S \ref{pmax}, we present the relationship between the Lesaint-Raviart-DG and BMS-DG schemes. We then discuss the relationship between Lesaint-Raviart-DG scheme and the characteristic methods. These are used to establish the positivity and the maximum principle. In \S \ref{num:ex}, we present a number of numerical tests to confirm our theory and provide a concluding remark in \S \ref{con}. 


%
%
\section{Governing Equations}\label{sec:intro}

In this section, we shall present the governing coupled flow and transports. We shall suppose $\Omega\subseteq \mathbb{R}^d$ to be a bounded polygon (d = 2) or polyhedron (d = 3) with Lipschitz boundary $\partial\Omega$. We consider the equation of conservation of mass
\begin{equation}\label{eqn:conservation_mass}
\nabla\cdot \bu = f \quad\mbox{ in } \Omega,   
\end{equation}
where $f$ is an external source/sink function such that at souce, $f > 0$ while $f < 0$ at sink. Here the velocity $\bu : \Omega \mapsto \mathbb{R}^d$ is defined by the Darcy's law:  
\begin{equation}\label{eqn:main_velocity}
\bu = -\dfrac{K}{\mu(c)}\nabla p \quad \mbox{ in } \Omega, 
\end{equation}    
where $p:\Omega \mapsto \mathbb{R}$ represents the pressure, 
$K$ is the permeability coefficient, and $\mu(c)$ is the fluid viscosity. Define $\bkappa := \bkappa(c) := K/\mu(c)$. For the boundary, we decompose $\partial \Omega$ into two parts $\Gamma_D$ and $\Gamma_N$ so that $\overline{\partial \Omega} = \overline{\Gamma}_D \cup \overline{\Gamma}_N$ and we then impose
\begin{subequations}\label{eqn:main_pressure_bc}
\begin{align}
p &= g_{_D}, \quad  \mbox{ on } \Gamma_D,  \\ 
\bu \cdot \bn &= g_{_N}, \quad \mbox{ on } \Gamma_N, 
\end{align}
\end{subequations}
where $\bn$ denotes the outward pointing unit normal to $\partial \Omega$, $g_D \in L^2(\Gamma_D)$ and $g_N \in L^2(\Gamma_N)$ are Dirichlet and Neumann boundary conditions, respectively.

We shall then consider the transport of chemically reactive species coupled with the flow equation \eqref{eqn:main_velocity} (see \cite{DawsonC_SunS_WheelerM-2004aa}). Transport equations for each chemical species are given as follows: 
\begin{equation}
\partial_t (\phi c_i) + \nabla \cdot (\bu c_i - D(\bu) \nabla c_i) = f c^*_i + R_i(c_1,\cdots,c_{n_c}), 
\end{equation} 
where $\partial_t$ is the time derivative, $c_i$ denotes the concentration of species $i=1,\cdots,n_c$ with $n_c$ being the number of chemical species, $D(\bu)$ is the velocity-dependent diffusion/dispersion tensor which is symmetric and positive semi-definite, $\phi$ is the volumetric factor such as porosity, and $R_i$ is a chemical reaction term. We note that the concentration $c^*_i=\widetilde{c}_i$ is specified at sources where $f > 0$, while it is unknown at sinks where $f < 0$. 
The advection and diffusion of chemical species are governed by a velocity field $\bu$, which can be shown to satisfy the continuity equation, conservation of mass equation given as follows, using the fact that $\sum_i c_i$ is constant and assuming that $\sum_i R_i = 0$. On the other hand, following \cite{DawsonC_SunS_WheelerM-2004aa}, we shall simply consider the single species case. Namely, we assume that the transport equation is given as follows: 
\begin{eqnarray}\label{eqn:full_incomp_a}
\partial_t c + \nabla \cdot (\bu c - D(\bu) \nabla c)  = fc^*, \quad &\mbox{ in } \Omega \times (0,\mathbb{T}], 
\end{eqnarray}
where $D(\bu)$ is assumed to be positive semi-definite. The boundary of $\Omega$ for transport system, denoted by $\partial \Omega$, is decomposed into two parts 
$\Gamma_{\rm I}$ and $\Gamma_{\rm O}$, the inflow and outflow boundary, respectively (i.e. $\overline{\partial \Omega} = \overline{\Gamma}_{\rm I}  \cup \overline{\Gamma}_{\rm O} $). They are defined as 
\begin{equation}
\Gamma_{\rm I} := \{ \bx \in \partial \Omega : \bu\cdot\bn < 0\} \quad \mbox{ and } \quad \Gamma_{\rm O} := \{ \bx \in \partial \Omega : \bu\cdot\bn \geq 0\},
\label{eqn:main_bd_def}
\end{equation}
where $\bn$ denotes the unit outward normal vector to $\partial \Omega$. 
For the boundary, we employ the following boundary conditions
\begin{subequations}\label{eq:bc_transport}
\begin{align}
(c\bu- D(\bu) \nabla c) \cdot \bn &= c_{\rm I} \bu \cdot \bn,&\mbox{ on } \Gamma_{{\rm I}} \times (0,\mathbb{T}],  \\
D(\bu)\nabla c \cdot \bn &= 0,  &\mbox{ on } \Gamma_{{\rm O}} \times (0,\mathbb{T}], 
\end{align}
\end{subequations}
where $c_{\rm I}$ is the inflow function. The initial conditions are set as $c(\bx,0) = c^0$. 

\subsection{Continuous Maximum Principle}\label{sec:mp}
In this section, we shall establish the maximum principle for the continuous case.  We shall see that the continuous maximum principle is strongly affected by the continuity equation: 
\begin{equation}
\nabla \cdot \bu = f. 
\end{equation}


First, we define a notation for a given function $\phi$ as follows:
\begin{equation}
\phi_{_+} :=\max\{\phi,0\} \quad \mbox{ and } \quad 
\phi_{_-} :=\max\{-\phi,0\}.
\end{equation}
Secondly, we show that the following continuous weak maximum principle holds under the condition that $\bu$ satisfies the strong conservation. 
\begin{theorem}\label{thm_wmp}
For the initial condition $c^0$, the boundary condition $c_{\rm I}$ and specified source $\widetilde{c}$, we let $M=\max\{|c^0|,|c_{\rm I}|,|\widetilde{c}|\}$. Under the condition of the strong conservation \eqref{eqn:conservation_mass}, we can show that the solution to equations \eqref{eqn:full_incomp_a}--\eqref{eq:bc_transport} satisfies $|c|\leq M$.
\end{theorem}
\begin{proof}
By the equation \eqref{eqn:full_incomp_a}, we have for $v \in H^1(\Omega)$, 
\begin{equation}\label{weak_form1} 
\int_\Omega \partial_t c \, v \, dx + \int_\Omega\nabla\cdot\bu c vdx +\int_\Omega \bu\cdot \nabla c \, v\,dx  - \int_\Omega \nabla\cdot (D(\bu) \nabla c) v\, dx 
= \int_\Omega fc^*vdx. 
\end{equation}
Using the integration by parts and the strong conservation  \eqref{eqn:conservation_mass}, we get 
\begin{equation}\label{eq:weakform}
\begin{aligned}
&\int_\Omega \partial_t c \, v \, dx  +\int_\Omega \bu\cdot \nabla c \, v\,dx  + \int_\Omega D(\bu) \nabla c \cdot \nabla v\, dx \\
&=\int_\Omega f(c^*-c)vdx+\int_{\partial \Omega}D(\bu)\nabla c\cdot\bn vdx.
\end{aligned}
\end{equation}
We shall show that $c \leq M$ by setting $v := (c-M)_{_+}$ in \eqref{eq:weakform}. We then observe that the following identity holds: 
\begin{equation}\label{eq4}
\begin{aligned}
\int_\Omega \bu \cdot \nabla c \, v\, dx &= \int_\Omega \bu\cdot \nabla c \, (c-M)_{_+} \, dx = \frac{1}{2} \int_\Omega \nabla [(c-M)_{_+}]^2 \cdot \bu\, dx \\
&=\frac{1}{2} \int_{\partial \Omega} v^2 \, \bu\cdot\bn \, ds-\frac{1}{2}\int_{\Omega}f v^2dx.
\end{aligned}
\end{equation}
Using this identity and the boundary condition $\eqref{eq:bc_transport}$, the equation \eqref{eq:weakform} can be shown to satisfy the following inequality: 
\begin{equation*}\label{weak_form2} 
\begin{aligned}
&\quad\frac{1}{2} \frac{d}{dt} \int_\Omega |( c - M)_{_+}|^2 \, dx + \int_\Omega D(\bu) |\nabla (c - M)_{_+}|^2 \, dx \\
& = \int_\Omega f(c^*-c)vdx +\int_{\partial \Omega} D(\bu) \nabla c\cdot \bn \, v\, ds - \frac{1}{2} \int_{\partial \Omega} v^2 \, \bu\cdot\bn \, ds + \frac{1}{2}\int_\Omega fv^2dx  \\ 
&  = \int_{\Gamma_{\rm I}} D(\bu) \nabla c\cdot\bn \, v \, ds  + \int_{\Gamma_{\rm O}} D(\bu) \nabla c\cdot\bn \, v \, ds + \int_\Omega f\left(c^*-c+\frac{1}{2}v\right)vdx \\ 
& \quad - \frac{1}{2} \int_{\Gamma_{\rm I}} [(c- M)_{_+}]^2 \, \bu \cdot \bn \, ds  - \frac{1}{2} \int_{\Gamma_{\rm O}} [(c- M)_{_+}]^2 \, \bu \cdot \bn \, ds  \\ 
&\leq \int_{\Gamma_{\rm I}} D(\bu) \nabla c\cdot\bn \, v \, ds - \frac{1}{2} \int_{\Gamma_{\rm I}} [(c- M)_{_+}]^2 \, \bu \cdot \bn \, ds+ \int_\Omega f\left(c^*-c+\frac{1}{2}v\right)vdx \\ 
&= \int_{\Gamma_{\rm I}} \left(c-c_{\rm I}-\frac12(c- M)_{_+}\right)(c- M)_{_+} \, \bu \cdot \bn \, ds\\
&\quad+\int_\Omega f\left(c^*-c+\frac{1}{2}(c-M)_{_+}\right)(c-M)_{_+} dx.  
\end{aligned}
\end{equation*}
We shall assume $c > M$. Then, we have 
\begin{equation*}
c-c_{\rm I}-\frac12(c- M)_{_+}=\frac{c}{2}+\frac{M}{2}-c_{\rm I}\ge M-c_{\rm I}\ge0.
\end{equation*}
Now, we consider the following identity: 
\begin{equation*} 
f\left(c^*-c+\frac{1}{2}(c-M)_{_+}\right)(c-M)_{_+} = \left \{  \begin{array}{l} 
f\left(\tilde{c}-c+\frac{1}{2}(c-M)_{_+}\right)(c-M)_{_+}, \,\, \mbox{ for } f > 0 \\
\frac{1}{2} f (c-M)_{_+} (c-M)_{_+},\,\,  \mbox{ for } f \leq 0.
\end{array} 
\right. 
\end{equation*}
This means 
\begin{equation*} 
f\left(c^*-c+\frac{1}{2}(c-M)_{_+}\right)(c-M)_{_+} \leq 0. \end{equation*} 
Therefore, we conclude that 
\begin{equation}
\frac{1}{2} \frac{d}{dt} \int_\Omega |( c - M)_{_+}|^2 \, dx \leq 0 
\end{equation}
and arrive that $c \leq M$ if $c^0 \leq M$. This is a contradiction to $c > M$ and it proves that $c \leq M$. The proof to show $c \ge -M$ can be done by setting $v := (c + M)_{_-}$ in \eqref{eq:weakform}. The rest of the process is similar and thus, we omit the details. This completes the proof that $|c|\leq M$.
\end{proof}

\section{Discontinuous Galerkin Finite Element Formulation for flow and transport}\label{dtransportflow}

In this section, we shall consider the discontinuous Galerkin finite element formulation of the flow coupled with the transport. Let $\mathcal{T}_h$ be the shape-regular triangulation by a family of partitions of $\O$ into $d$-simplices $\K$ (triangles/squares in $d=2$ or tetrahedra/cubes in $d=3$), $h_T$ denote the diameter of $T$ and $h = \max_{T \in \mathcal{T}_h} h_T$. 
Also we denote by $\Eh$ the set of all edges, and by $\Eho$ and $\Ehb$ the collection of all interior and boundary edges, respectively. 

We shall also introduce standard tools such as jumps and averages of scalar and vector valued functions across the edges of $\mathcal{T}_h$. Let $e$ be an interior edges shared by two elements $T^\pm$. We define the unit normal vectors $\bn^\pm$ on $e$ pointing exterior to $T^\pm$, respectively. For a function $\phi$ that is piecewise smooth on $\mathcal{T}_h$, with $\phi^\pm = \phi|_{T^\pm}$, we define 
\begin{equation}
\av{\phi} = \frac{1}{2} (\phi^+ + \phi^-) \quad \mbox{ and } \quad \jump{\phi} = \phi^+ \bn^+ + \phi^-\bn^- \quad \forall e \in \mathcal{E}_h^o,
\end{equation}
where $\mathcal{E}_h^o$ is the set of interior edges $e$. For a vector-valued function $\bm{\tau}$, piecewise smooth on $\mathcal{T}_h$, with $\bm{\tau}^\pm = \bm{\tau}|_{T^\pm}$, we define 
\begin{equation}
\av{\bm{\tau}} = \frac{1}{2} (\bm{\tau}^+ + \bm{\tau}^-) \quad \mbox{ and } \quad \jump{\bm{\tau}} = \bm{\tau}^+ \bn^+ + \bm{\tau}^-\bn^- \quad \forall e \in \mathcal{E}_h^o.
\end{equation}
For a set of boundary edges $e \in \mathcal{E}_h^\partial$, let 
\begin{equation}
\jump{\phi} = \phi \bn, \quad \mbox{ and } \quad \av{\bm{\tau}} = \bm{\tau}, \quad \forall e \in \mathcal{E}_h^\partial.
\end{equation}
The space $H^{s}(\Th)$ $(s\in \mathbb{R})$ is the set of element-wise $H^{s}$ functions on $\mathcal{T}_h$ with $s \geq 1$, and $L^{2}(\Eh)$ refers to the set of functions whose traces on the elements of $\Eh$ are square integrable. Let $\mathbb{P}_k(\K)$ denote the space of polynomials of degree at most $k$ on $T$.
Define the space of piecewise discontinuous polynomials of degree $k$ by
\begin{equation}
V_{h,k} := \left \{ p \in L^2(\Omega) | \ p_{|_{\K}} \in \mathbb{P}_k(\K), \ \forall \K \in \mathcal{T}_h \right \}. 
\end{equation}
We also use the following notations:
\begin{alignat*}{2}
(v,w)&:=\dyle\sum_{\K \in \Th} \int_{\K} v\, w dx \qquad &&\forall\,\, v ,w \in L^{2} (\mathcal{T}_h), \\
\langle v, w\rangle_{\Eh}&:=\dyle\sum_{e\in \Eh} \int_{e} v\, w \,ds \qquad &&\forall\, v, w \in L^{2}(\Eh).
\end{alignat*}

\subsection{DG formulation for the flow and the concept of the local conservation}
In this section, we use $V_{h,k_p}$ as the DG space for the pressure approximation and present the DG formulation for the flow equation, which is given as follows: 
\begin{subequations}\label{eq:flow}
\begin{align}
\bu &= -\kappa \nabla p, \quad \mbox{ in } \Omega, \\
\nabla \cdot \bu &= f, \quad \mbox{ in } \Omega,
\end{align}
\end{subequations} 
which is subject to the boundary conditions \eqref{eqn:main_pressure_bc}. The interior penalty DG formulation \cite{wheeler1978elliptic,RiviereB_WheelerM_GiraultV-2001aa} for \eqref{eq:flow} is then given as: Find $P \in V_{h,k_p}$ such that
\begin{equation}\label{floweq}
\begin{aligned}
&\quad(\kappa \nabla P, \nabla w) - \langle \av{\kappa \nabla P}, \jump{w} \rangle_{\mathcal{E}_h} + \langle \sigma \jump{P}, \jump{w} \rangle_{\mathcal{E}_h^o} + \langle \sigma (P-g_D), w \rangle_{\Gamma_D} \\
&\quad + \theta_f  \langle \av{\kappa \nabla w}, \jump{P}  \rangle_{\mathcal{E}_h^o} +\theta_f \langle \kappa \nabla w \cdot\bn, (P -g_D) \rangle_{\Gamma_D} =  (f,w), \quad \forall w \in V_{h,k_p}, 
\end{aligned}
\end{equation}
where $\sigma$ is the penalty parameter and it is given as $\sigma = O(1/h)$. Note that the parameter $\theta_f$ determines the type of DG, namely, $\theta_f = 0$ for IIPG, $\theta_f = -1$ for SIPG and $\theta_f = 1$ for NIPG  \cite{de2012l2}. 
We shall now discuss how to define the numerical flux. In fact, it is standard to define the flux using $P$, the solution to \eqref{eq:flow} in the following manner: 
\begin{subequations}\label{fluxeq}
\begin{align}
{\bf{U}} &=  -\kappa \nabla P, \quad \mbox{ in } T,~\forall T \in \mathcal{T}_h, \\
{\bf{U}} &= -\av{\kappa \nabla P} + \sigma \jump{P}, \quad \mbox{ on } e,~\forall e \in \mathcal{E}_h^o, \\
{\bf{U}}\cdot \bn &= g_N, \quad \mbox{ on } e, ~\forall e \in \Gamma_N, \\
{\bf{U}}\cdot \bn &= -(\kappa \nabla P) \cdot\bn + \sigma (P - g_D),~\mbox{ on } e, \quad\forall e \in \Gamma_D, 
\end{align}
\end{subequations}
This means that the flux $U$ satisfies the following equation: 
\begin{eqnarray}\label{fluxeq1} 
&& - ({\bf{U}}, \nabla w) + \langle {\bf{U}}, \jump{w} \rangle_{\mathcal{E}_h^o} + 
\langle {\bf{U}}\cdot \bn, w \rangle_{\partial \Omega} + \theta_f \langle \av{\kappa \nabla w}, \jump{P}  \rangle_{\mathcal{E}_h^o} \\
&&  +\theta_f \langle \kappa \nabla w \cdot\bn, (P -g_D) \rangle_{\Gamma_D} =  (f,w), \quad \forall w \in V_{h,k_p}.  \nonumber 
\end{eqnarray} 
Necessarily, the flux satisfies the standard local conservation. 
\begin{equation}
\int_{\partial T} {\bf{U}} \cdot \bn \, ds = \int_T f \, dx. 
\end{equation}
On the other hand, if $\theta_f = 0$, then the flux ${\bf{U}}$ satisfies the following equation: 
\begin{subeqnarray*}
- ({\bf{U}}, \nabla w) + \langle {\bf{U}}, \jump{w} \rangle_{\mathcal{E}_h^o} + 
\langle {\bf{U}}\cdot \bn, w \rangle_{\partial \Omega}
 =  (f,w), \quad \forall w \in V_{h,k_p}.    
\end{subeqnarray*} 
Equivalently, for each $T \in \mathcal{T}_h$, the flux $\bU$ satisfies the following equation:   
\begin{equation}\label{conser}
\int_T \nabla \cdot {\bf{U}} w \, dx = \int_T f w \, dx, \quad \forall w \in V_{h,k_p}. 
\end{equation}
We can view this identity as a stronger version of the local conservation since the standard local conservation is the case when the equation \eqref{conser} holds for $w$ being the piecewise constant function. Furthermore, if $k_p$ is replaced by $k_c$, then it has been discussed as the compatibility condition in \cite{DawsonC_SunS_WheelerM-2004aa}, which will be discussed in more details in \S \ref{compx} below. Thus, we are motivated to introduce this condition as a new concept of the local conservation. 

We are in a position to state the main definition of the new local conservation, i.e., the local conservation of degree $k$.  
\begin{definition}\label{def} 
The flux $\bU$ is said to be locally conservative of degree $k$ with respect to the triangulation $\mathcal{T}_h$ if and only if for all $T \in \mathcal{T}_h$, it holds that 
\begin{equation} 
\int_T \nabla \cdot \bU \, w \, dx = \int_T f\, w \, dx, \quad \forall w \in \mathbb{P}_{k}(T). 
\end{equation} 
\end{definition} 
In \S \ref{pmax}, we discuss the importance of the local conservation, which include $L^2$norm stability, zeroth order accuracy, positivity and maximum principle preserving. We remark that if $\nabla \cdot \bU = f$, i.e., strong conservation holds, then such a local conservation is necessarily true.

\subsection{DG formulation for the transport}\label{BMS}
In this section, we present the DG formulation for transport, primarily based on the work of Brezzi, Marini, and Süli \cite{brezzi2004discontinuous} for the hyperbolic part of the transport. The elliptic part follows the standard interior penalty method by Wheeler \cite{wheeler1978elliptic}, similar to the flow case. 

In our discussion, we shall assume that the mesh is sufficiently small or the flux $\bU$ is sufficiently smooth so that 
\begin{equation}
\bU \cdot \bn_e|_e \mbox{ does not change its sign for all } e \in \mathcal{E}_h,  
\end{equation}
where $\bn_e$ is the normal to the edge $e$. This is crucial to establish the relationship between Lesaint-Raviart DG \cite{LasaintP_RaviartP-1974aa} and Brezzi-Marini-S{\"u}li DG. Under this assumption, for any given edge $e \in \mathcal{E}_h^o$, which is the common edge of two triangles $T_1$ and $T_2$, we can determine which triangle is upstream triangle $T^-$ and which is the downstream triangle $T^+$. Specifically, if $\bU \cdot \bn^1 \geq 0$, then $T_1=T^-$ is the upstream triangle, and $T_2=T^+$ is the downstream triangle. Otherwise, we denote $T_1 = T^+$ and $T_2 = T^-$. For $e \in \mathcal{E}_h^\partial$, if $e \in \Gamma_-$, then $T = T^+$ while if $e \in \Gamma_+$, then $T = T^-$. 

We shall consider the transport equation \eqref{eqn:full_incomp_a} and boundary condition \eqref{eq:bc_transport} with $\bu$ replaced by $\bU$,  i.e., 
\begin{equation}\label{eq:transport}
\partial_t c + \nabla \cdot ( \bU c - D \nabla c) = f c^*, \quad \mbox{ in } \Omega \times (0,\mathbb{T}], 
\end{equation} 
subject to the boundary conditions 
\begin{subequations} 
\begin{align}
(c\bU - D\nabla c) \cdot \bn &= c_{\rm I} \bU \cdot \bn, \quad \Gamma_{\rm I} \times (0,\mathbb{T}] \\
D\nabla c \cdot \bn &= 0, \quad \Gamma_{\rm O} \times (0,\mathbb{T}]. 
\end{align} 
\end{subequations}
By applying a simple Euler time discretization and Brezzi-Marini-S{\"u}li DG to \eqref{eq:transport}, we have the following Euler-DG formulation: 
Given the old time level solution $C^{\rm old}$, find $C \in V_{h,k_c}$ such that 
\begin{equation}\label{transport-DG}
\begin{aligned}
&\gamma (C, w) - ({\bf{U}} C - D \nabla C, \nabla w) 
+ \left \langle \left ( \av{C {\bf{U}}} + c_e \jump{C} - \av{D\nabla C} \right ), \jump{w} \right \rangle_{\mathcal{E}_h^o}  \\
& + \theta_c \langle \av{D\nabla w}, \jump{C} \rangle_{\mathcal{E}_h^o} + \langle c_{\rm I} {\bf{U}}\cdot \bn, w\rangle_{\Gamma_{\rm I}} + \langle C{\bf{U}}\cdot\bn, w \rangle_{\Gamma_{\rm O}} + \langle \sigma_D \jump{C}, \jump{w} \rangle_{\mathcal{E}_h^o} \\
& = (fC^*, w) + (g,w), \qquad \forall w \in V_{h,k_c},  
\end{aligned}
\end{equation}
where $g = \gamma C^{\rm old}$ with $\gamma = \dfrac{1}{\Delta t}$, $c_e = |\bU \cdot \bn|/2$, the jump stabilized term, as discussed in \cite{brezzi2004discontinuous}, and the parameter $\theta_c$ determines the type of DG scheme for the diffusion part of the transport, i.e., for $\theta_c = 0$, it gives IIPG, while $\theta_c = -1$ and $\theta_c = 1$ give SIPG and NIPG, respectively. $\sigma_D$ is the penalty parameter for the diffusive term. Note that $\sigma_D = 0$ if $D = 0$. In particular, for the pure hyperbolic case $D = 0$, the equation \eqref{transport} reduces to 
\begin{equation}\label{transport}
\begin{aligned}
&\gamma (C, w) - ({\bf{U}} C,  \nabla w) 
+ \left \langle \left ( \av{C {\bf{U}}} + c_e \jump{C} \right ), \jump{w} \right \rangle_{\mathcal{E}_h^o}  \\
&+ \langle c_{\rm I} {\bf{U}}\cdot \bn, w\rangle_{\Gamma_{\rm I}} + \langle C{\bf{U}}\cdot\bn, w \rangle_{\Gamma_O} 
= (fC^*, w) + (g,w), \quad \forall w \in V_{h,k_c},
\end{aligned}
\end{equation}
which can also be written as the following formulation:  
\begin{equation}\label{upwind} 
\begin{aligned}
&\sum_{T \in \mathcal{T}_h} \int_T \gamma (C w - C (\bU \cdot \nabla w)\, dx + \sum_{e \not \subset \Gamma_-} \int_e \av{\bU C}_u \cdot \jump{w} \, ds \\
&=\int_\Omega fC^* w \, dx + \int_\Omega g\,w\, dx - \sum_{e \subset \Gamma_-} \int_e (\bU \cdot \bn) c_{\rm I} w \, ds, \quad \forall w \in V_{h,k_c},
\end{aligned}
\end{equation}
where $\av{\bU C}_u=\av{C {\bf{U}}} + c_e \jump{C}$ represents the upwind value of $\bU C$. More precisely, it is defined as 
across the edge of an element over which it is evaluated. We recall the definition of the upwind value for $\bU C$, given through
\begin{eqnarray}
\av{\bU C}_u = \left \{ \begin{array}{ll} 
\bU C^+ & \mbox{ if } \bU \cdot \bn^+ > 0 \\
\bU C^- & \mbox{ if } \bU \cdot \bn^- > 0 \\
\bU \av{C} & \mbox{ if } \bU \cdot \bn^\pm = 0. 
\end{array} \right. 
\end{eqnarray}

\subsection{Compatibility condition and its relationship to the local conservation}\label{compx}

In this section, we shall review the compatibility condition pioneered by Dawson, Sun and Wheeler \cite{DawsonC_SunS_WheelerM-2004aa}. Basically, the compatibility condition is the minimial requirement imposed on the flux approximation to obtain meaningful solution to the transport. The meaningfulness of the solution to the transport was found at two requirements. One is the global conservation and the other is the zeroth-order accuracy. For the schemes in this paper, the requirement of the global conservation is generally independent of how $\bU$ is computed \cite{DawsonC_SunS_WheelerM-2004aa}. Thus, this does not seem to be closely related to the compatibility condition.  Rather, it is related to the zeroth-order accuracy. The zeroth-order accuracy can be considered to be a weaker version of the maximum or minimum principle. Namely, if the transport scheme can not reproduce a constant solution, it means there arise spurious sources and sinks due to numerical inaccuracies. In relation to DG formulation \eqref{upwind} of the transport, the zeroth order accuracy is shown to hold under the assumption that the numerical flux satisfies the following compatibility condition \cite{DawsonC_SunS_WheelerM-2004aa}: 
\begin{equation}
- ({\bf{U}}, \nabla w)_\Omega + \langle {\bf{U}}, \jump{w}\rangle_{\mathcal{E}_h^o} + \langle {\bf{U}}\cdot\bn, w \rangle_{\partial\Omega} = (f,w)_{\Omega}, \quad \forall w \in V_{h,k_c}. 
\end{equation} 
Note that this is obtained by setting $C=C^*=c_{\rm I}=c^0$ is a constant in the equation \eqref{transport} and it is equivalent to the following statement, i.e., for all $T \in \mathcal{T}_h$, it holds that 
\begin{equation}
\int_T \nabla \cdot \bU \, w \, dx = \int_T f \, w \, dx, \quad \forall w \in V_{h,k_c}. 
\end{equation} 
We observe that the compatibility condition is the local conservation of degree $k_c$. This observation is summarized in the following theorem: 
\begin{theorem}
For $k_c > 0$, the local conservation of degree $k_c$ can be achieved by using $k_p \geq k_c$ for the flow equation with $\theta_f = 0$ only. For $k_c = 0$, the local conservation of degree $k_c$ can be achieved by using $k_p \geq k_c$ with $\theta_f = -1, 0, 1$.
\end{theorem}
\begin{proof} 
We recall that the flux $\bU$ satisfies the equation \eqref{fluxeq1}. Thus, in order to obtain the local conservation of degree $k_c$, it is clear to use $k_p \geq k_c$. For $k_c > 0$, only $\theta_f = 0$ can lead to the local conservation of degree $k_c$. On the other hand, if $k_c = 0$, then any $\theta_f = -1, 0, 1$ can lead to the local conservation of degree $k_c = 0$. This completes the proof. 
\end{proof} 

\section{$L^2$-Stability, Positivity and Discrete maximum principle (DMP)}\label{pmax} 
In this section, assuming that $\bU$ is locally conservative of degree $k$, we shall establish the $L^2$ stability, positivity and maximum principle of the solution to transports. For simplicity and within the scope of our presentation, we shall restrict our focus to the non-diffusive transports case. Namely, we assume that the transport is given by the following linear hyperbolic equation:  
\begin{equation}\label{hyper} 
\partial_t c + \nabla \cdot (\bU c) = fc^*, \quad \mbox{ in } \quad \Omega\times(0,\mathbb{T}], 
\end{equation} 
subject to 
\begin{equation}
c = c_{\rm I}, \quad \mbox{ on } \quad \Gamma_- = \{ x \in \partial \Omega : \bU \cdot \bn < 0 \}. 
\end{equation}
We will start the analysis by exploring two technical tools. The first one is the equivalence relationship between the Lesaint-Raviart DG method and the BMS DG method for the hyperbolic equation \eqref{hyper}. The second one is the relationship between the Lesaint-Raviart DG method and the characteristic method along the streamline.  

\subsection{The Lesaint-Raviart DG method and its equivalence with the BMS-DG scheme} 

In this section, we shall introduce the Lesaint-Raviart DG formulation for the equation \eqref{hyper} and show that it is equivalent to the BMS-DG formulation. 

We note that similar to the streamline diffusion methods \cite{BrooksA_HughesT-1982aa,DawsonC_SunS_WheelerM-2004aa}, the Lesaint-Raviart DG formulation is based on the following reformulation of \eqref{hyper}:
\begin{equation}\label{eq: euler_nondivergence}
\partial_t c+ \bU \cdot\nabla c + (\nabla\cdot\bU)c = fc^*. 
\end{equation}
After time discretization by Euler's method, the Lesaint-Raviart formulation is designed using the discontinuous Galerkin finite element: Find $c_h \in V_{h,k_c}$ such that $c_h = c_{\rm I}$ on $\Gamma_-$ and satisfies 
\begin{equation}\label{lrform} 
\begin{aligned}
&\sum_{T \in \mathcal{T}_h} \int_T (\bU \cdot \nabla c_h)w_h + (\nabla \cdot \bU) c_h w_h + \gamma c_h w_h \, dx \\
& - \sum_{T \in \mathcal{T}_h} \int_{\Gamma_-^{\circ}(T)}(c_h^+ - c_h^-) w_h^+ (\bU \cdot \bm{n}^+) \, ds + \sum_{e\in \Gamma_-}\int_e c_h |\bU \cdot\bn| w_h \, ds \\
&= \sum_{T \in \mathcal{T}_h} \int_T (fc^* + g) w_h\, dx + \sum_{e\in \Gamma_-}\int_e c_{\rm I} |\bU \cdot \bn|w_h\, ds, \quad \forall w_h \in V_{h,k_c}, 
\end{aligned}
\end{equation}
where $g=\gamma c_h^{old}$ with $\gamma=\frac{1}{\Delta t}$, $\Gamma_-^{\circ}(T)$ represents the interior edge with $\bU\cdot \bn<0$ and $u^\pm = \lim\limits_{\epsilon \rightarrow 0^+} u(\bm{x} \pm \epsilon \bU)$. 
The following lemma basically establishes the equivalence relationship between the BMS-DG method and Lesaint-Raviart DG method.
\begin{lemma}\label{lrbms}
It holds true that 
\begin{equation}
\begin{aligned}
&\sum_{T \in \mathcal{T}_h} \int_{\partial T} c_h w_h (\bU\cdot \bm{n})\, ds - \sum_{T \in \mathcal{T}_h} \int_{\Gamma_-^{\circ}(T)} (c_h^+ - c_h^-) w_h^+ (\bU\cdot \bm{n}^+) \, ds\\
&-\int_{\Gamma_-} c_h\bU\cdot\bn w_h \, ds
= \sum_{e \in \mathcal{E}_h^o} \int_e \av{c_h} \cdot \jump{w_h} \,ds + \sum_{e \in \mathcal{E}_h^o}\int_e \frac{|\bU\cdot \bm{n}|}{2} \jump{c_h} \cdot \jump{w_h} \, ds. 
\end{aligned}
\end{equation}
\end{lemma}
\begin{proof}
It is well known that \cite{arnold2002unified}
\begin{equation}
\sum_{T \in \mathcal{T}_h} \int_{\partial T} (\bU \cdot \bm{n})c_h w_h \, ds = \sum_{e \in \mathcal{E}_h} \int_e \av{\bU c_h}\jump{w_h} \, ds + \sum_{e \in \mathcal{E}_h^o} \int_e \jump{\bU c_h} \av{w_h}\, ds. 
\end{equation}
Thus, we have that 
\begin{equation*}
\begin{aligned}
&\sum_{T \in \mathcal{T}_h} \int_{\partial T} c_h w_h (\bU\cdot \bm{n})\, ds - \sum_{T \in \mathcal{T}_h} \int_{\Gamma_-^{\circ}(T)} (c_h^+ - c_h^-) w_h^+ (\bU\cdot \bm{n}^+) \, ds -\int_{\Gamma_-} c_h\bU\cdot\bn w_h \, ds \\
&= 
\sum_{e \in \mathcal{E}_h} \int_e \av{\bU c_h}\jump{w_h} \, ds + \sum_{e \in \mathcal{E}_h^o} \int_e \jump{\bU c_h} \av{w_h}\, ds-\int_{\Gamma_-} c_h\bU\cdot\bn w_h \, ds\\
&\quad- \sum_{T \in \mathcal{T}_h} \int_{\Gamma_-^{\circ}(T)} (c_h^+ - c_h^-) w_h^+ (\bU\cdot \bm{n}^+) \, ds. 
\end{aligned}
\end{equation*}
Since $\jump{\phi} = \phi \bm{n}$ and $\av{\bm{\tau}} = \bm{\tau}$ on $e \in \mathcal{E}_h^\partial$, we see that 
\begin{equation}
\sum_{e \in \mathcal{E}_h}  \int_e\av{\bU c_h} \jump{w_h} \,ds = \sum_{e \notin \Gamma_-} \int_e\av{\bU c_h} \jump{w_h} \,ds + \sum_{e \in \Gamma_-}\int_e (\bU \cdot \bm{n} ) gw_h \, ds. 
\end{equation}
Thus, we only need to show that 
\begin{equation}
\begin{aligned}
&\sum_{e \in \mathcal{E}_h^o} \int_e \jump{\bU c_h} \av{w_h} \, ds - \sum_{T \in \mathcal{T}_h} \int_{\Gamma_-^{\circ}(T)} (c_h^+ - c_h^-) w_h^+ (\bU\cdot\bm{n}^+) \, ds \\
& = \sum_{e \in \mathcal{E}_h^o} \int_e  \frac{|\bU\cdot\bm{n}|}{2} \jump{c_h}\cdot\jump{w_h}\, ds. 
\end{aligned}
\end{equation}
For any given $e \in \mathcal{E}_h^o$, we see that
\begin{eqnarray*}
&& \quad\jump{\bU c_h}\av{w_h} - (c_h^+ - c_h^-)w_h^+ \bU \cdot \bn^+ \\
&& =( \bU c_h^+ \cdot \bn^+ + \bU c_h^- \cdot \bn^- ) \frac{w_h^+ + w_h^-}{2} 
 - (c_h^+ - c_h^-) w_h^+ \bU \cdot \bm{n}^+ \\
&&  = -\frac{1}{2}  \bU c_h^+w_h^+ \bm{n}^+ + \frac{1}{2} \bU c_h^+w_h^- \bm{n}^+ -\frac{1}{2}  \bU c_h^-w_h^+ \bm{n}^- + \frac{1}{2} \bU c_h^-w_h^- \bm{n}^- \\
&&  = \frac{|\bU\cdot\bm{n}|}{2} \jump{c_h} \cdot \jump{w_h}. 
\end{eqnarray*}
This completes the proof. 
\end{proof}
By using the integral by parts for the first term in \eqref{lrform}, we can easily find that the Lemma \eqref{lrbms} shows the equivalence between the BMS-DG scheme \cite{brezzi2004discontinuous} and Lesaint-Raviart DG scheme \cite{LR1974}. Due to the equivalence relationship, we can conclude that $C = c_h$. 

\subsection{The relation between the Lesaint-Raviart DG method and the Characteristic method}

In this section, we shall investigate the relationship between the characteristic methods and the Lesaint-Raviart DG, to solve the nondivergence form of equation \eqref{eq: euler_nondivergence}, or \eqref{hyper}. This is crucial to prove the positivity and maximum principle of the BMS DG formulation. We first introduce an instrumental lemma.
\begin{lemma}
Assume that the flux satisfies the local conservation of degree $k\geq 2k_c$. Then the system \eqref{lrform} reduces to the following: 
\begin{equation}
B(C, w_h) = F(w_h), \quad \forall w_h \in V_{h,k_c},  
\end{equation}
where 
\begin{subeqnarray*}
B(C, w_h) &:=& \sum_{T \in \mathcal{T}_h} \int_T (\bU \cdot \nabla C + f_{_+} C + \gamma C) w_h \, dx \\
&& \,\, + \sum_{e \in \mathcal{E}_h^o} \int_e \jump{C}\bn^+ w_h^+ |\bU \cdot \bn|\, ds + \sum_{e \in \Gamma_-} \int_e C  |\bU \cdot \bn| w_h\, ds \nonumber \\
F(w_h) &:=& \int_\Omega  f_{_+} \widetilde{c} w_h \, dx + \sum_{e \in \Gamma_-} \int_e c_{\rm I} |\bU \cdot \bn| w_h\, ds + \int_\Omega g w_h \, ds, 
\end{subeqnarray*}
where $f_{_+}$ is the source and $g = \gamma C^{\rm old}$ with $\gamma = \frac{1}{\Delta t}$. 
\end{lemma}
\begin{proof} 
Since $\bU$ is locally conservative with the degree $k\geq 2k_c$, we get
\begin{equation}
\int_\Omega \nabla \cdot \bU C w_h \, dx = \int_\Omega f C w_h \, dx, \quad \forall C, w_h \in V_{h,k_c}. 
\end{equation}
By using the fact that $c^*=c$ for $f<0$, we can conclude that $B(C,w_h)=F(w_h)$. 
\end{proof}
The associated norm is given as follows:
\begin{equation}
\|w\|_h^2 = \|w\|_{0,\Omega}^2 + \sum_{e \in \mathcal{E}_h^o} \int_e \jump{w}^2 |\bU \cdot \bn| \, ds + \sum_{e \in \Gamma_-} \int_{e} w^2 |\bn \cdot \bU| \, ds, \quad \forall w \in V_{h,k_c}.  
\end{equation}
We note that the bilinear form $B$ satisfies the coercivity with respect to the above energy norm, i.e., $\|\cdot\|_h$ \cite{LR1974}. The BMS DG can also be shown to be stable with respect to $\|\cdot\|_h$ \cite{brezzi2004discontinuous}. 

We shall now consider the Characteristic method following \cite{baranger1997existence,bermudez1987methode}. We introduce a parameter $\eta$, i.e., the pseudo time step size and denote $S(x,t,\tau)$ the solution of differential system of characteristics:
\begin{equation}\label{chara1}
\frac{dS}{d\tau} = \bU(S) \quad \mbox{ subject to } S(x,t,t) = x. 
\end{equation}
We note that since $\bU \in W^{1,\infty}$, the Cauchy-Lipschitz theorem \cite{solomon1977differential} leads to the unique existence of the solution. We shall denote the characteristic feet of $x$ by $X^\eta(x)$, i.e., 
\begin{equation} 
X^\eta(x) = S(x,t,t-\eta), \quad \forall \eta \in \Reals{}.  
\end{equation}
Note that $X^\eta(x)$ is the position of $x$ at the previous time following the velocity $\bU$ for $\eta > 0$ i.e., the upstream position of $x$ for $\eta > 0$.   
Under the local conservation of degree $k = 2k_c$, the variational formulation of characteristic method is given by finding $c^\eta \in L^2(\Omega)$ such that 
\begin{equation}\label{chara} 
B_\eta(c^\eta,w) = F_\eta(w), \quad \forall w \in L^2(\Omega), 
\end{equation} 
where 
\begin{subeqnarray*}
B_\eta(c,w) &:=& \frac{1}{\eta} \int_{\Omega} c w \,dx - \frac{1}{\eta} \int_{\Omega_1} c(X^\eta(x)) w \, dx + \int_\Omega (f_{_+}c + \gamma c) w \, dx \\
F_\eta(w) &:=& \int_\Omega f_{_+} \widetilde{c} w \, dx + \frac{1}{\eta} \int_{\Omega_2} c_{\rm I} w \, dx + \int_\Omega g w\,dx ,  
\end{subeqnarray*}
with $\Omega_1 \cup \Omega_2 = \Omega$ being such that 
\begin{subeqnarray*}
\Omega_{1} = \{x \in \Omega : X^\eta(x) \in \Omega \} \quad \mbox{ and } \quad  \Omega_{2} = \{x \in \Omega : X^\eta(x) \notin \Omega \}. 
\end{subeqnarray*}

We now consider to solve the equation \eqref{chara} using DG with the approximation space given by $V_{h,k_c}$. The discrete problem can then be read as to find $C^\eta \in V_{h,k_c}$ such that 
\begin{eqnarray}\label{cDG} 
B_\eta(C^\eta,w_h) = F_\eta(w_h), \quad \forall w_h \in V_{h,k_c}.  
\end{eqnarray}
To investigate the well-posedness of the equation \eqref{cDG}, the special energy norm denoted by $\|\cdot\|_{h,\eta}$ has been introduced in \cite{baranger1999natural}, i.e., 
\begin{equation}
\begin{aligned}
&\|w\|_{h,\eta}^2 := \int_{\Omega} w^2 \, dx + \frac{1}{\eta} \int_{\Omega_2} w^2 (x) \, dx + \frac{1}{\eta} \int_{\Omega_1} \left( w (x) - w (X^\eta (x)) \right)^2 dx \\
& \quad\quad\qquad+\int_{\Omega_3} \frac{w^2 (X^\eta (x))}{\eta} \, dx,
\quad \forall w \in V_{h,k_c}. 
\end{aligned}
\end{equation}
where $\Omega_3 = X^{-\eta}(\Omega)\backslash \Omega$. The bilinear form $B_\eta$ is shown to be coercive with respect to the energy norm $\|\cdot\|_{h,\eta}$ and thus, the well-posedness was established in \cite{baranger1999natural}. Most importantly, the following results were established in \cite{baranger1999natural}
\begin{theorem}\label{l2stability}
Assume that $\bU \in W^{1,\infty}(\Omega)$. Then 
\begin{equation}
\begin{aligned}
\lim_{\eta \rightarrow 0} B_\eta(u,v) &= B(u,v), \quad \forall u, v \in V_{h,k_c},    \\
\lim_{\eta \rightarrow 0} F_\eta(v) &= F(v), \quad\forall v \in V_{h,k_c},   
\end{aligned}
\end{equation}
and 
\begin{equation}
\lim_{\eta \rightarrow 0} \|u\|_{h,\eta} = \|u\|_h, \quad \forall u \in V_{h,k_c}. 
\end{equation}
\end{theorem}
The natural consequence of this theorem is that $C_\eta$ converges to $C$ as $\eta \rightarrow 0$. Namely, we have
\begin{equation}\label{main:limit} 
\lim_{\eta \rightarrow 0} C^\eta = C. 
\end{equation}

We shall use this result to prove $L^2$ norm stability, positivity and the maximum principle. For the $L^2$ norm stability, we first introduce the weighted $L^2$ norm, which is defined by 
\begin{equation}
\|u\|_{L^2_\eta}=\left(\int_\Omega \left(1+\eta f_{_+}+\eta\gamma\right)u^2\,dx\right)^{\frac{1}{2}}.
\end{equation}

\begin{theorem}
Assume that the flux satisfies the local conservation of degree $k\geq 2k_c$ and $M=\max\{|c^0|, |c_{\rm I}|,|\widetilde{c}|\}$, then the solution $C^\eta$ to \eqref{chara} satisfies the $L^2_\eta$ stability:
\begin{equation}
\|C^\eta\|_{L^2_\eta}\leq \|M\|_{L^2_\eta}. 
\end{equation}
\end{theorem}
\begin{proof}
We shall first invoke the fixed point iteration to obtain $c^\eta_h$. Set $0<c^{\eta,0}_h = C^{\rm old}<M$ and consider the following iterates $\{C^{\eta,\ell}\}_{\ell = 0,1,\cdots}$ defined by the following rule: 
\begin{equation}\label{iteration}
\begin{aligned}
& \frac{1}{\eta} \int_{\Omega} C^{\eta,\ell+1} w_h \,dx - \frac{1}{\eta} \int_{\Omega_1} C^{\eta,\ell}(X^\eta(x)) w_h \, dx + \int_\Omega (f_{_+} + \gamma) C^{\eta,\ell+1} w_h \, dx \\
& = \int_\Omega f_{_+} \widetilde{c} w_h \, dx + \frac{1}{\eta} \int_{\Omega_2} c_{\rm I} w_h \, dx + \int_\Omega g w_h\, dx.  
\end{aligned}
\end{equation}
We recall that $C^\eta$ solves the following equation:
\begin{equation}
    \begin{aligned}
& \frac{1}{\eta} \int_{\Omega} C^{\eta} w_h \,dx - \frac{1}{\eta} \int_{\Omega_1} C^{\eta}(X^\eta(x)) w_h \, dx + \int_\Omega (f_{_+} + \gamma) C^{\eta} w_h \, dx \\
&= \int_\Omega f_{_+} \widetilde{c}w_h \, dx + \frac{1}{\eta} \int_{\Omega_2} c_{\rm I} w_h \, dx + \int_\Omega g w_h \, dx.  
\end{aligned}
\end{equation}
By subtracting two equations, setting $w_h = e_h^{\ell+1} = C^{\eta,\ell+1} - C^\eta$ and multiplying $\eta$ to both sides, we get
\begin{eqnarray*}
(1 + \eta \min_{x \in \Omega} (f_{_+} + \gamma) )\int_{\Omega} |e_{h}^{\ell+1}|^2 \,dx &=& \int_{\Omega_1} e_h^{\ell}(X^\eta(x)) e_h^{\ell+1} \, dx \\ 
&\leq& \left ( \int_{\Omega_1} |e_h^\ell(X^\eta)|^2 \, dx \right )^{\frac{1}{2}} \left (\int_\Omega |e_h^{\ell+1}|^2 \,dx \right )^{\frac{1}{2}}. 
\end{eqnarray*}
Thus, we have the following inequality since $X^\eta$ is one-to-one and $\Omega_1 \subset \Omega$, 
\begin{equation} 
(1 + \eta \min_{x \in \Omega}(f_{_+} + \gamma)) \|e_{h}^{\ell+1}\|_{0,\Omega} \,dx \leq  \|e_h^\ell\|_{0,\Omega_1} \leq \|e_h^\ell\|_{0,\Omega}.  
\end{equation}
Thus, we prove the convergence of the fixed point iteration. 
Let 
\begin{equation*}
\widetilde{C}^{\eta,\ell}(x)=
\left\{
\begin{aligned}
&C^{\eta,\ell}(X^\eta(x)), && x\in\Omega_1\\ 
&c_{\rm I}, && x\in\Omega_2
\end{aligned}
\right.
\end{equation*}
Then we have 
\begin{equation}\label{iteration1}
\begin{aligned}
&\frac{1}{\eta} \int_{\Omega} C^{\eta,\ell+1} w_h \,dx  + \int_\Omega (f_{_+} + \gamma) C^{\eta,\ell+1} w_h \, dx \\
& = \int_\Omega f_{_+} \widetilde{c} w_h \, dx + \frac{1}{\eta} \int_{\Omega} \widetilde{C}^{\eta,\ell} w_h \, dx + \int_\Omega g w_h\, dx.  
\end{aligned}
\end{equation}
Take $w_h=C^{\eta,\ell+1}$ in the above equation, we get
\begin{equation}\label{eq: w=c}
\begin{aligned}
&\quad\frac{1}{\eta} \int_{\Omega} (C^{\eta,\ell+1})^2 \,dx  + \int_\Omega (f_{_+} + \gamma) (C^{\eta,\ell+1})^2 \, dx \\
& = \int_\Omega f_{_+} \widetilde{c} C^{\eta,\ell+1} \, dx + \frac{1}{\eta} \int_{\Omega} \widetilde{c}_h^{\eta,\ell} C^{\eta,\ell+1} \, dx + \int_\Omega g C^{\eta,\ell+1}\, dx\\
& \leq \frac{1}{2}\int_\Omega f_{_+}|\widetilde{c}|^2\, dx +\frac{1}{2}\int_\Omega f_{_+}|C^{\eta,\ell+1}|^2\, dx + \frac{1}{2\eta}\int_\Omega |\widetilde{C}^{\eta,\ell}|^2\, dx + \frac{1}{2\eta}
\int_\Omega |C^{\eta,\ell+1}|^2\, dx\\ 
&\quad+ \frac{1}{2}\int_\Omega \gamma |C^{old}|^2\,dx + \frac{1}{2}\int_\Omega \gamma|C^{\eta,\ell+1}|^2\, dx 
\end{aligned}
\end{equation}
by using the Cauchy-Schwarz inequality and the Young's equation. Therefore, 
\begin{equation}
\begin{aligned}
    &\quad \int_\Omega \left(1+\eta f_{_+}+\eta\gamma\right) |C^{\eta,\ell+1}|^2\, dx\\
    &\leq \int_\Omega \eta f_{_+}|\widetilde{c}|^2\, dx +\int_\Omega |\widetilde{C}^{\eta,\ell}|^2\, dx+\int_\Omega \eta\gamma |C^{old}|^2\,dx\\
    &\leq \int_\Omega \left(1+\eta f_{_+}+\eta\gamma\right)M^2\,dx.
    \end{aligned}
\end{equation}
Thus, $\|C^{\eta,\ell+1}\|_{L^2_\eta}\leq \|M\|_{L^2_\eta}$.
By the convergence of this iterative method, we get the result $\|C^{\eta}\|_{L^2_\eta}\leq \|M\|_{L^2_\eta}$.   
\end{proof}
\begin{theorem}
Assume that the flux satisfies the local conservation of degree $k\geq 2k_c$, and suppose that the problem \eqref{hyper} has the nonnegative inflow $c_{\rm I}$ and initial $c^0$ condition, $M=\max\{|c^0|, |c_{\rm I}|, |\widetilde{c}|\}$. Then the numerical solution $C$ obtained by BMS-DG satisfies the $L^2$ stability. 
\end{theorem}
\begin{proof}
We can easily find that $\lim\limits_{\eta\to 0}\|u\|_{L^2_\eta}=\|u\|_{L^2}$. Thus, 
by combining the Theorem \eqref{l2stability} and the equation \eqref{main:limit}, we can conclude that 
\begin{equation}
\|C\|_{L^2} \leq M|\Omega|^{\frac{1}{2}}.
\end{equation}
This completes the proof. 
\end{proof}
Next, we shall assume that $k_c = 0$ and prove the positivity and the maximum principle. 
\begin{theorem}\label{main:lem}
Assume that the flux satisfies the local conservation of degree $k\geq 2k_c=0$. The solution $C^\eta$ to \eqref{cDG} satisfies $0\leq C^\eta\leq M$, where 
\begin{equation}
M = \max \{|c^{0}|, |c_{\rm I}|,|\widetilde{c}| \}.
\end{equation}
\end{theorem}
\begin{proof} 
Consider piecewise constant function in \eqref{iteration1}, define $w_h=1$ in $T$ and $w_h=0$ in others, then we have
\begin{equation}\label{iteration2constant}
\begin{aligned}
&\frac{1}{\eta} \int_T C^{\eta,\ell+1} \,dx  + \int_T (f_{_+} + \gamma) C^{\eta,\ell+1} \, dx \\
& = \int_T f_{_+} \widetilde{c}  \, dx + \frac{1}{\eta} \int_{T} \widetilde{C}^{\eta,\ell}(x) \, dx + \int_T g\, dx.  
\end{aligned}
\end{equation}
Obviously, we can get $C^{\eta,\ell+1}\geq 0$ by induction.
Additionally, the above equation yields
\begin{equation}\label{iteration3}
\begin{aligned}
\int_T \left ( \frac{1}{\eta}+f_{_+} + \gamma \right ) C^{\eta,\ell+1}\, dx &= \int_T f_{_+} \widetilde{c}  \, dx + \frac{1}{\eta} \int_{T} \widetilde{C}^{\eta,\ell} \, dx + \int_T g\, dx  \\
&\leq \int_T  \left( \frac{1}{\eta}+f_{_+} + \gamma \right )M\,dx.
\end{aligned}
\end{equation}
Then we have, 
\begin{equation*}
\begin{aligned}
\int_T  \left( \frac{1}{\eta}+f_{_+} + \gamma \right )(C^{\eta,\ell+1}-M)\,dx\leq 0.
\end{aligned}
\end{equation*}
Thus, $C^{\eta,\ell+1}\leq M$ since $\dfrac{1}{\eta}+f_{_+} + \gamma>0$. Using the fact that the fixed point iteration converges, we can conclude that $0\leq C^\eta\leq M$.
\end{proof} 
Combined with the above theorem and \eqref{main:limit}, we can obtain the positivity and the discrete maximum principle.
\begin{theorem}
Assume that the flux satisfies the local conservation of degree $k\geq 2k_c=0$ and suppose that the problem \eqref{hyper} has the nonnegative inflow $c_{\rm I}$ and initial condition $c^0$, $M=\max\{|c^0|, |c_{\rm I},|\widetilde{c}|\}$. Then the numerical solution $C$ obtained by BMS-DG satisfies the positivity and the discrete maximum principle. 
\end{theorem}

\section{Numerical Experiments}\label{num:ex} 

In this section, we present a simple numerical experiment that confirms our theory. 
We shall consider to solve the transport equation \eqref{eqn:full_incomp_a} with $\bu$ governed by the Darcy's law \eqref{eqn:main_velocity}. 
Here we assume $\mu(c) = 1$. Unless otherwise specified, we use the IIPG methodin in this section, i.e., $\theta_f=0$.

\subsection{1D case of propagating a concentration front through the domain}
We first consider the one-dimensional transport equation without the diffusion term:
\begin{equation} 
\partial_t c + (uc)_x = fc^*, \quad 0 < x < 1, \mbox{ with } t > 0, 
\end{equation} 
where $f = u_x = -\pi/2 \sin(\pi x/2)$ is the sink, that is $c^* = c$. Note that there is no source in this case. For the flow equation, we choose $K=1$ with Dirichlet boundary condition given as $p(0)=0$ and $p(1)=-2/\pi$. Thus, the flow direction is from left to right, we let  the inflow concentration $c_{\rm I}=1$ and the initial condition $c^0=0.1$. 

On the one hand, we study the behavior of the solution and the maximum principle. We discretize the interval $[0,1]$ with $10, 20, 50, 100$ and $200$ grid blocks, respectively. For the time step, we use $\Delta t=0.02$ at $T=0.5$. Use the piecewise linear function for the flow equation with penalty parameter $\sigma=100/h$.
From the first two figures in Fig. \ref{fig: 1D}, we can find that this method give virtually identical transport solutions, namely a traveling front two constant regions where $c=1$ and $c=0.1$. However, when using a piecewise linear function for the transport, we not only observe oscillations in the middle, but also encounter nonphysical values in the cases with $10$ and $20$ points. 
On the other hand, we test the $L^2$-stability. For the last plot in Fig. \ref{fig: 1D}, we use the piecewise quadratic function for the transport equation and $DG_1, DG_2, DG_3$ and $DG_4$ for the flow equation to compute $U$ with 10 points. It is evident that the $DG_1$ flux does not satisfy the $L^2$-stability criterion, as it fails to meet the local conservation. This observation is consistent with our theoretical predictions.
\begin{figure}
    \centering
\includegraphics[width=0.3\textwidth]{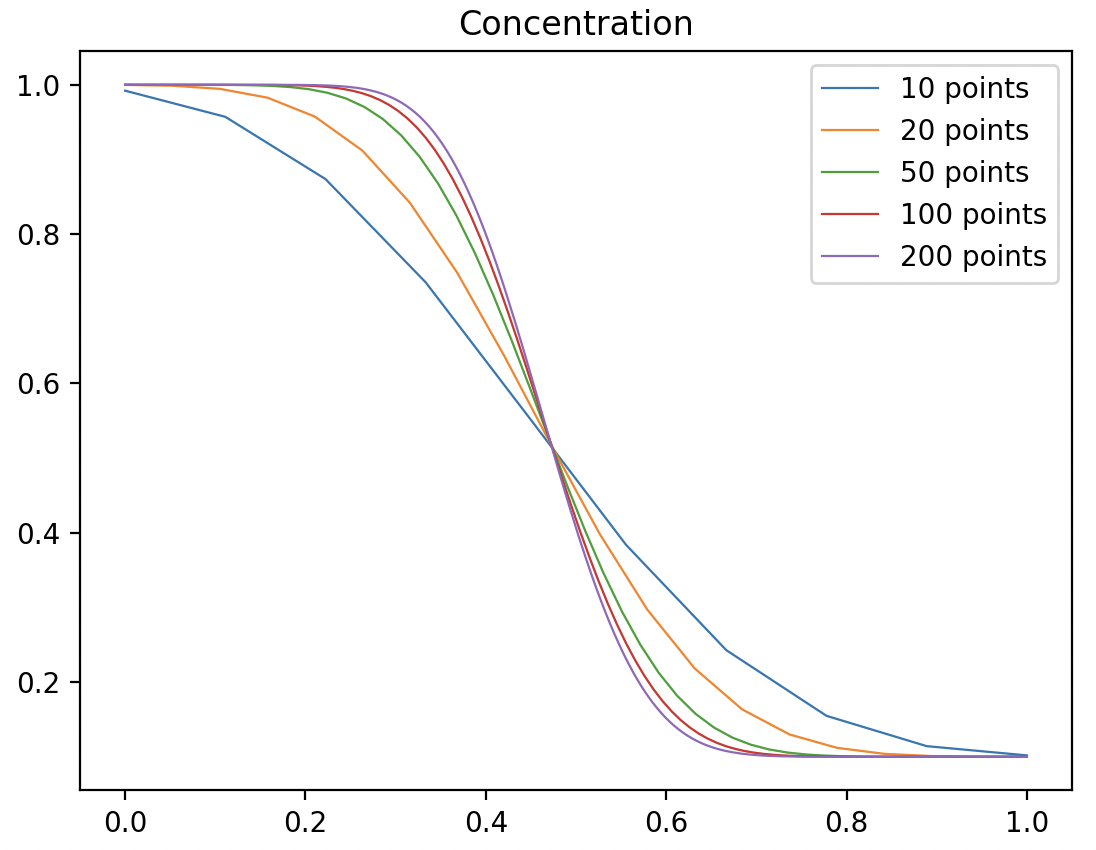}
\includegraphics[width=0.3\textwidth]{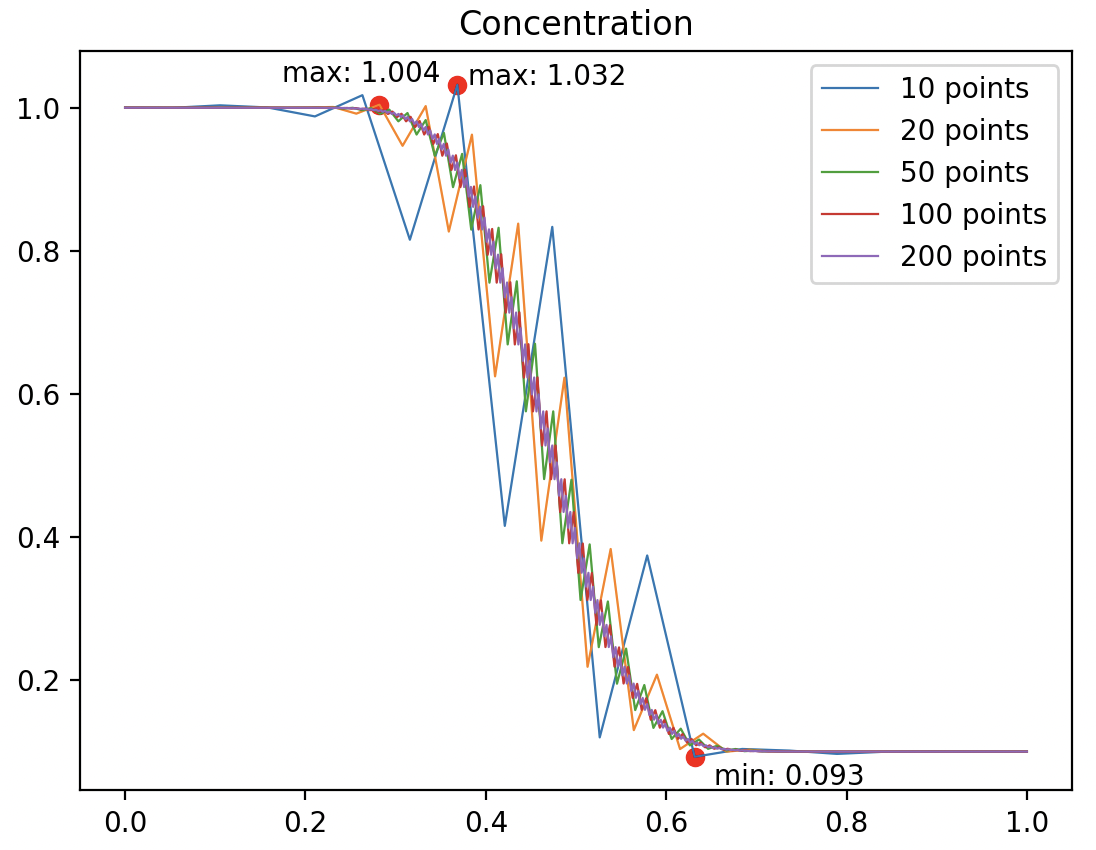}
\includegraphics[width=0.3\textwidth]{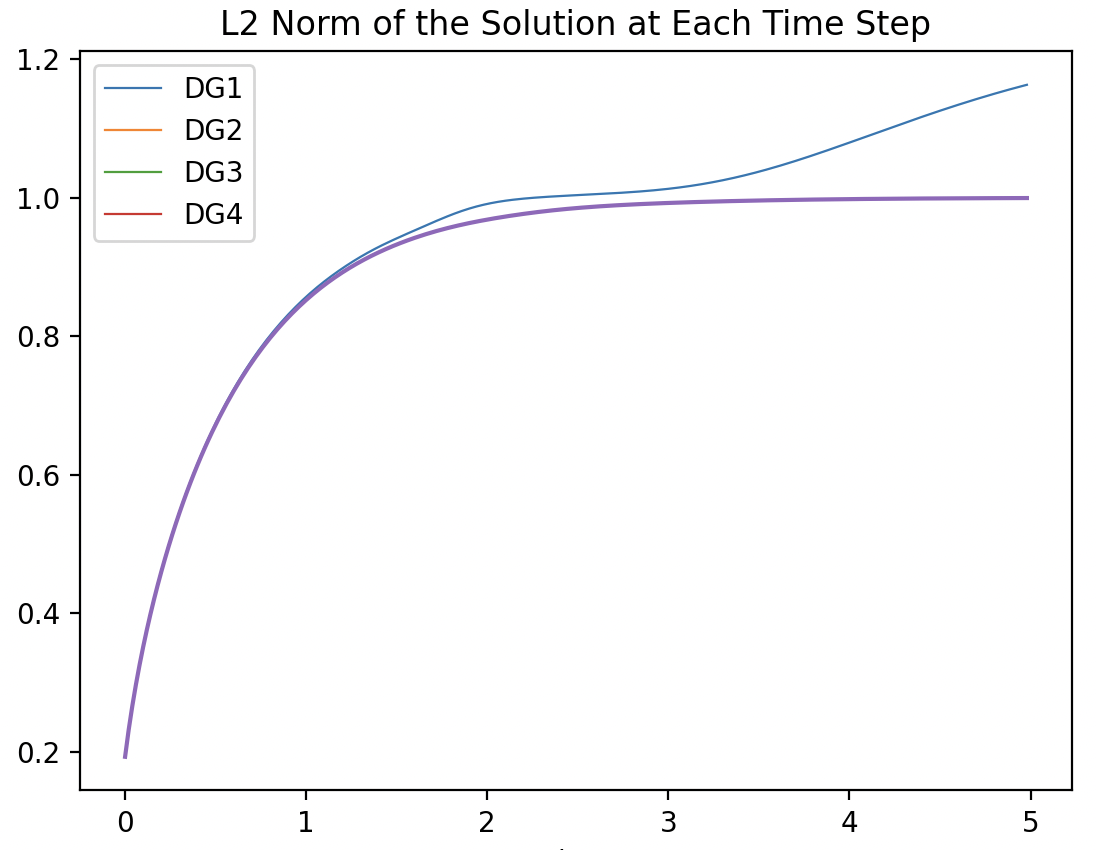}
    \caption{Concentration values at $T=0.5$ with piecewise constant function(left) and piecewise linear function(middle); $L^2$ norm for each time step using $DG_2$ for the transport equation and $DG_1$, $DG_2$, $DG_3$ and $DG_4$ for the flow equation.(right). }
    \label{fig: 1D}
\end{figure}


\subsection{2D case with no external force and no diffusion}

A simple example is the permeability block. Let $\Omega = (0,1)^2$, and the permeability tensor $\kappa$ is defined as a diagonal tensor defined as follows: 
\begin{equation}
\begin{aligned}
\bm{\kappa} = \left \{ \begin{array}{l} \left ( \begin{array}{cc} 10^{-3} & 0 \\ 0 & 10^{-3} \end{array} \right ) \mbox{ in } \Omega_s = \left (\frac{3}{8}, \frac{5}{8}  \right ) \times \left(\frac{1}{4}, \frac{3}{4} \right) \\ 
 \left ( \begin{array}{cc} 1 & 0 \\ 0 & 1 \end{array} \right )  \mbox{ elsewhere. } \end{array} \right.
 \end{aligned}
\end{equation} 
The inflow boundary condition for $c$ is given as 
$c_{\rm I} = 1 \mbox{ on } \Gamma_{D1} \times (0, T],$  
subject to the initial condition $c(0,x) = 0$ in $\Omega$. The pressure boundary condition is given as  $p = 1 \mbox{ on }  \Gamma_{D1}$, $p = 0 \mbox{ on } \Gamma_{D2}$ and $\bu \cdot \bn = 0 \mbox{ on } \Gamma_N$
For simplicity, we shall set $f = D(\bu) = 0$. 
The block and the solution to the flow equation solved by $DG_1$ can be seen in Fig. \ref{fig:sample}. After applying a piecewise constant function to solve the transport equation, we obtain the results depicted in Figure \ref{fig: block}, showcasing $T=0.5, 1, 1.5,$ and $3$ from left to right. These results adhere to the maximum principle and maintain positivity.
\begin{figure}[h!]
\centering
\includegraphics[scale=.16]{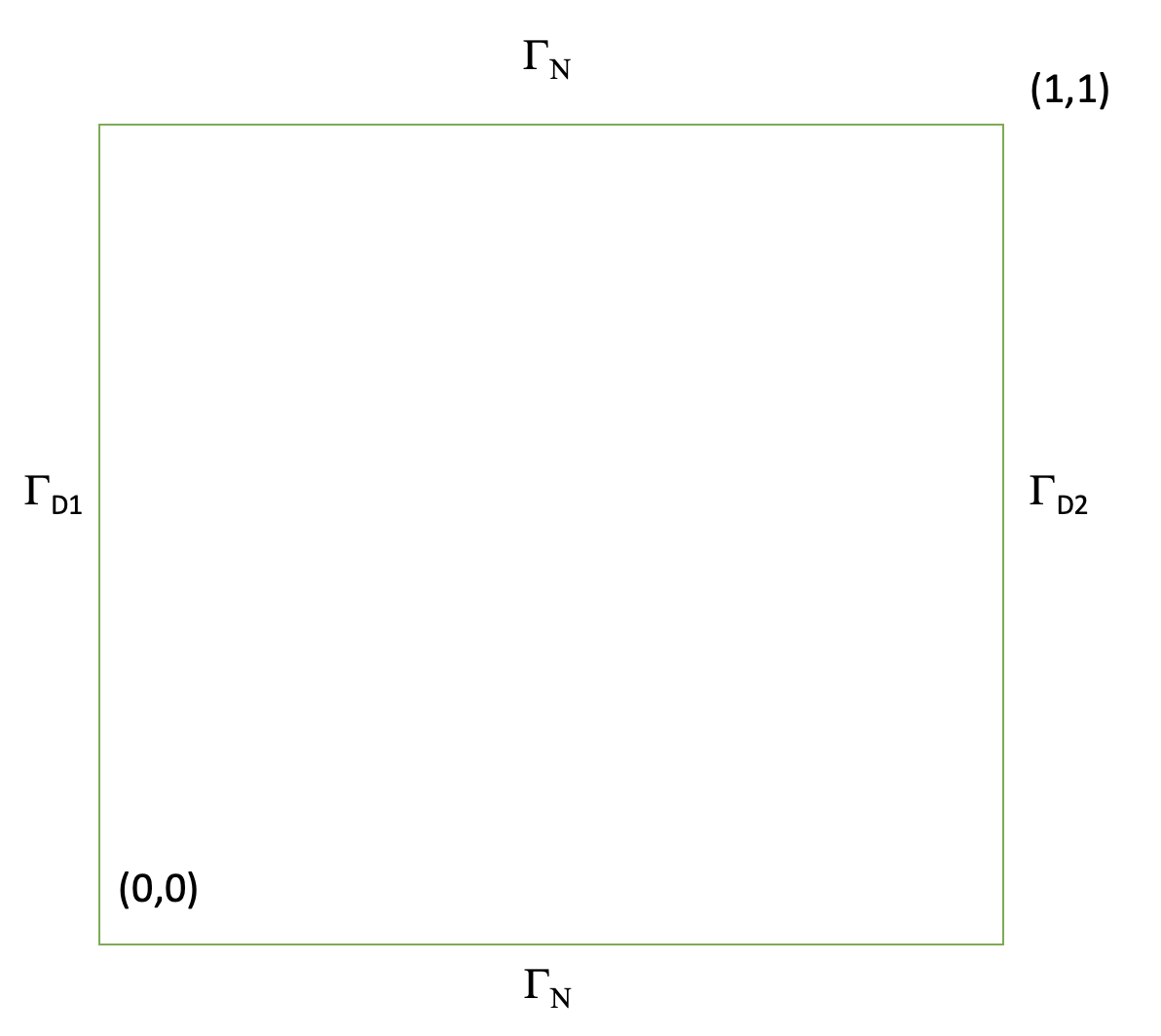} 
\hspace{0.1cm} 
\includegraphics[scale=.12]{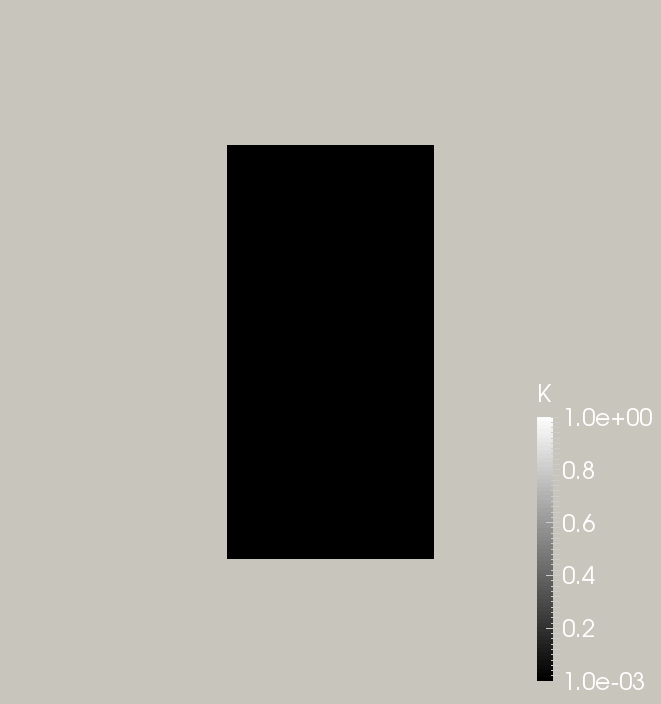}
\hspace{0.1cm} 
\includegraphics[scale=.12]{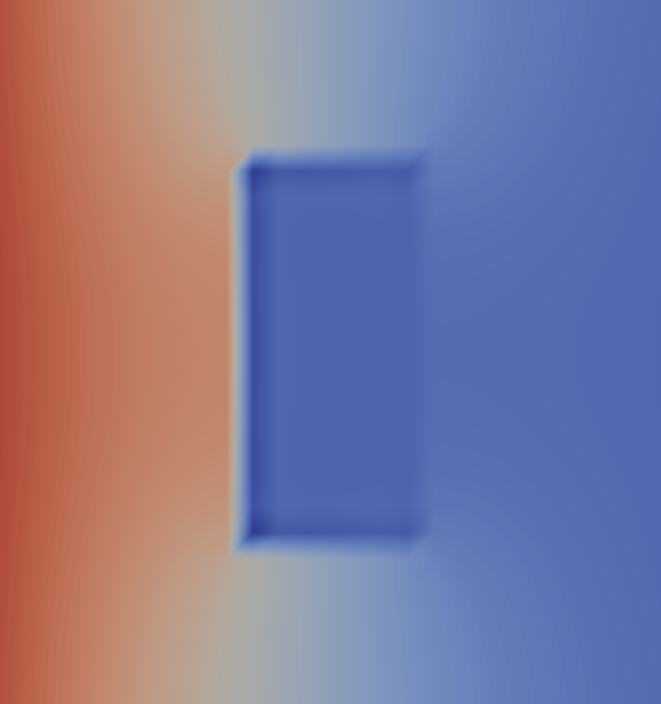}
\caption{Left: Domain; Middle: K-block; Right: Pressure.}
\label{fig:sample}
\end{figure}
\begin{figure}
    \centering
\subfigure[$T=0.5$]{\includegraphics[width=0.22\textwidth]{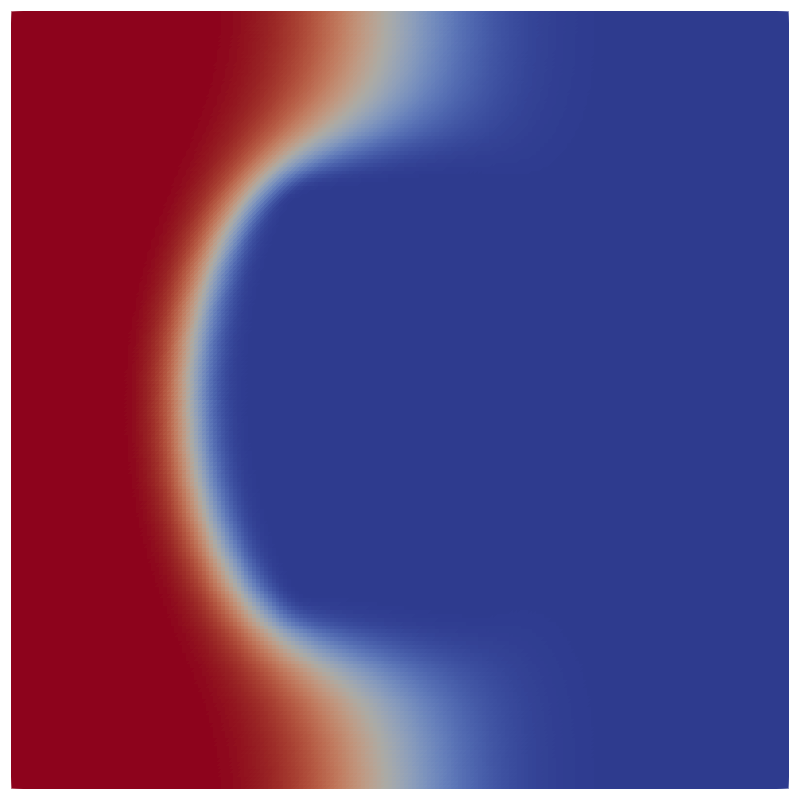}}
\subfigure[$T=1.0$]{\includegraphics[width=0.22\textwidth]{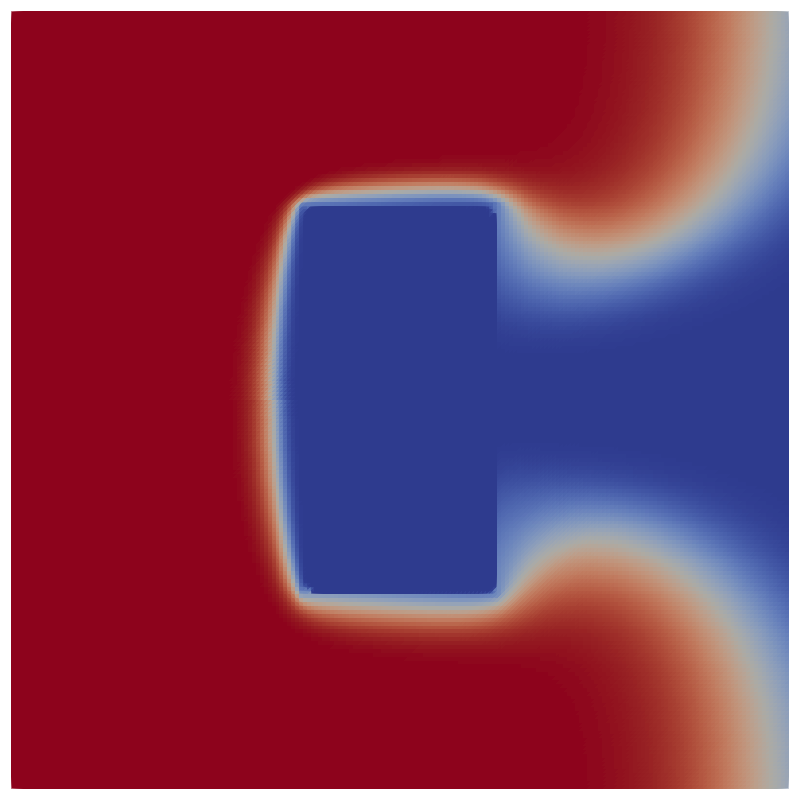}}
\subfigure[$T=1.5$]{\includegraphics[width=0.22\textwidth]{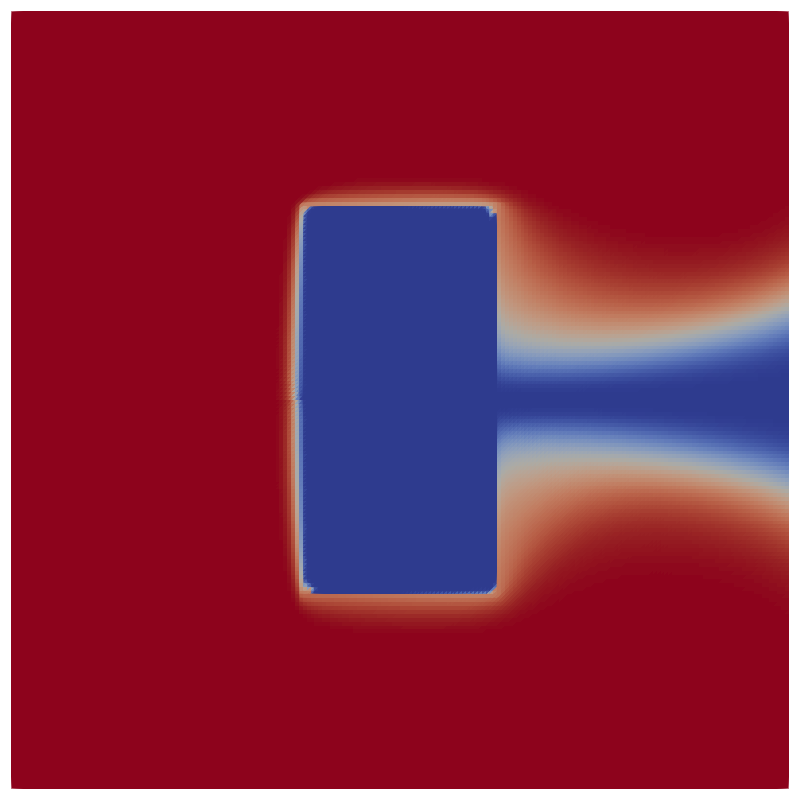}}
\subfigure[$T=3.0$]{\includegraphics[width=0.22\textwidth]{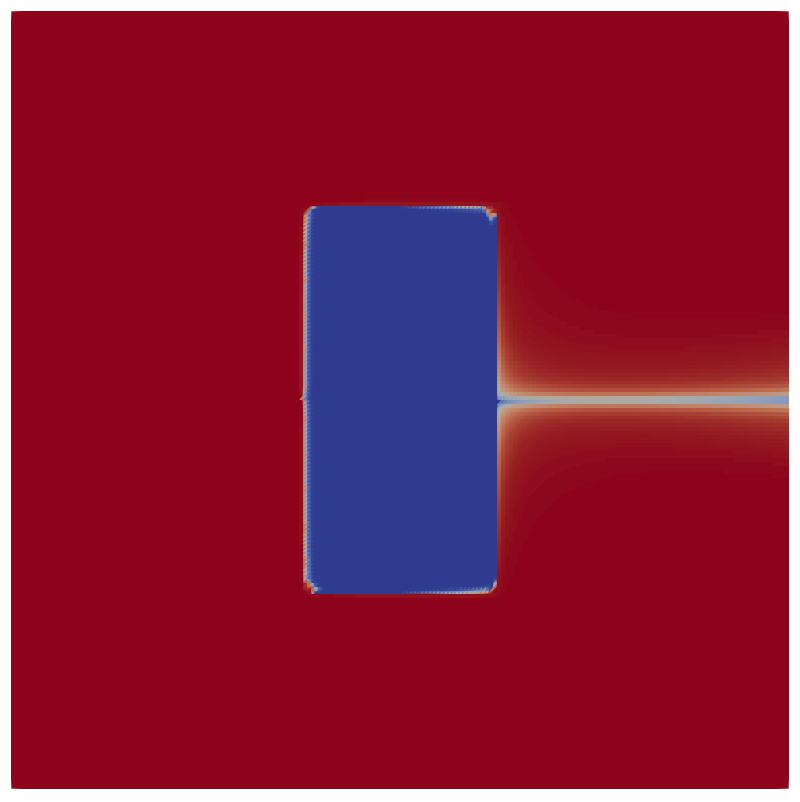}}
\includegraphics[width=0.07\textwidth]{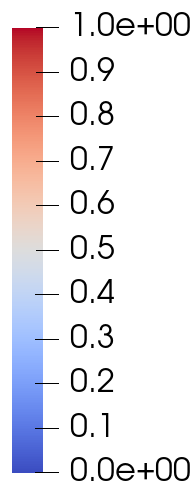}
    \caption{Concentration values at $T=0.5, 1.0, 1.5, 3.0$ (from left to right).}
    \label{fig: block}
\end{figure}

\subsection{2D case with external force for injection and extraction} 

We now consider the nonzero external force case. Let $\Omega = (0,1)^2$. The diagonal entry for the permeability $\bm{\kappa}$ is defined by: 
\begin{equation}
\bm{\kappa} = \left \{ \begin{array}{l} 
\left ( \begin{array}{cc} 10^{-3} & 0 \\ 0 & 10^{-3} \end{array} \right ) \mbox{ for } x \geq 0.5 \\ 
\left ( \begin{array}{cc} 1 & 0 \\ 0 & 1 \end{array} \right )  \mbox{ elsewhere. } \end{array} \right. 
\end{equation} 
The fluid is injected at the corner $(0,0)$ with the injection rate $f^+=100$, and it is extracted at the corner $(1,1)$ with the extraction rate $f^-=-100$. Using the piecewise linear function for the flow equation and the piecewise constant function for the transport equation, we set $\Delta t=h=0.01$. In Figure \ref{fig: pointsource}, we can see that the fluid smears into the low permeability zone and towards the extraction point, results in a physical values. If we use the pieciwise linear function for transport equation, solutions greater than 1 will occur, once again highlighting the importance of local conservation.
\begin{figure}
    \centering
\includegraphics[width=0.22\textwidth]{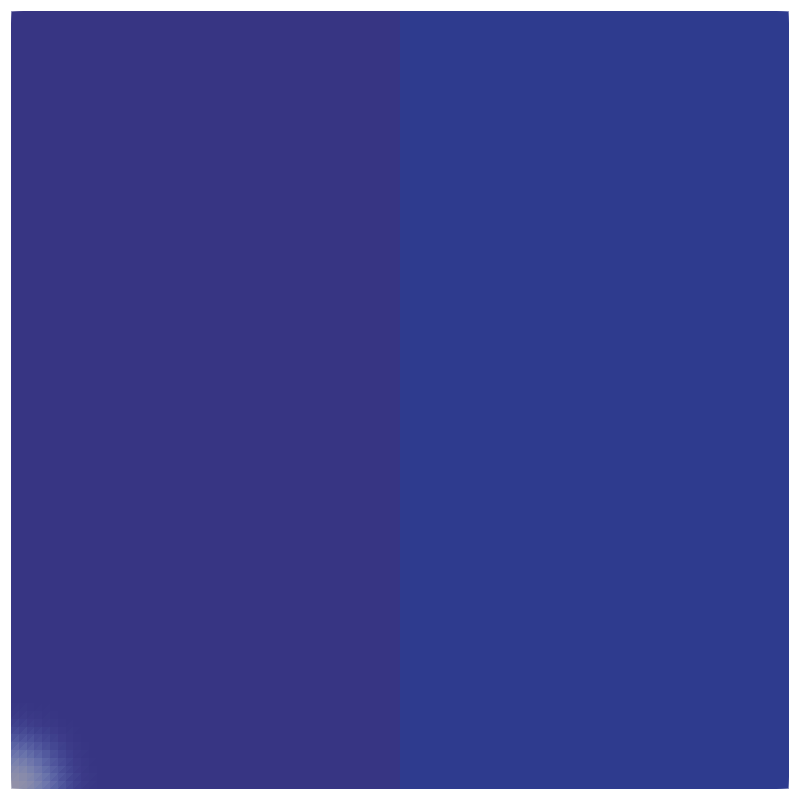}
\includegraphics[width=0.22\textwidth]{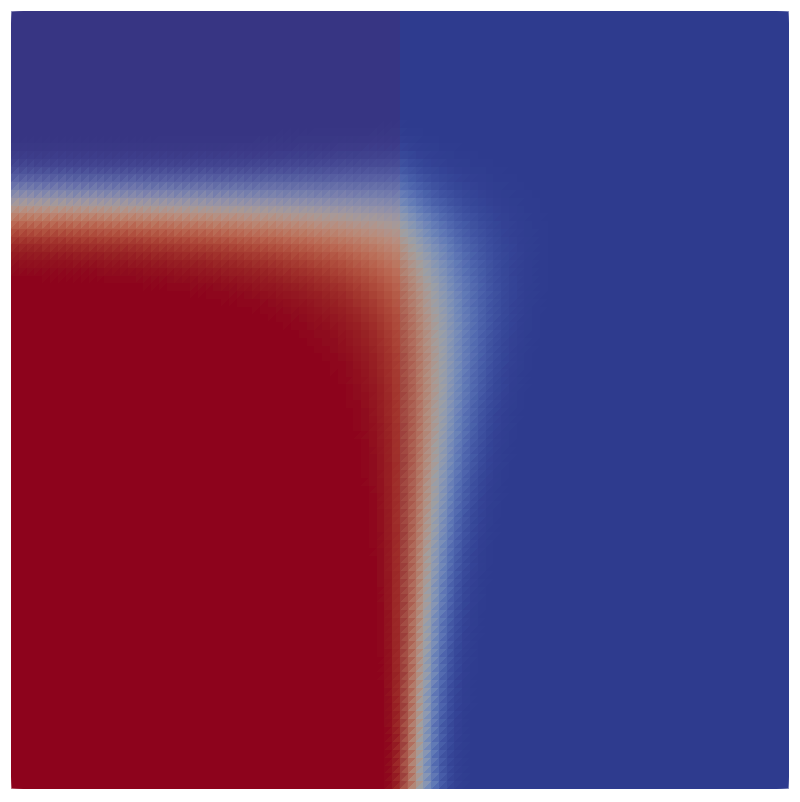}
\includegraphics[width=0.22\textwidth]{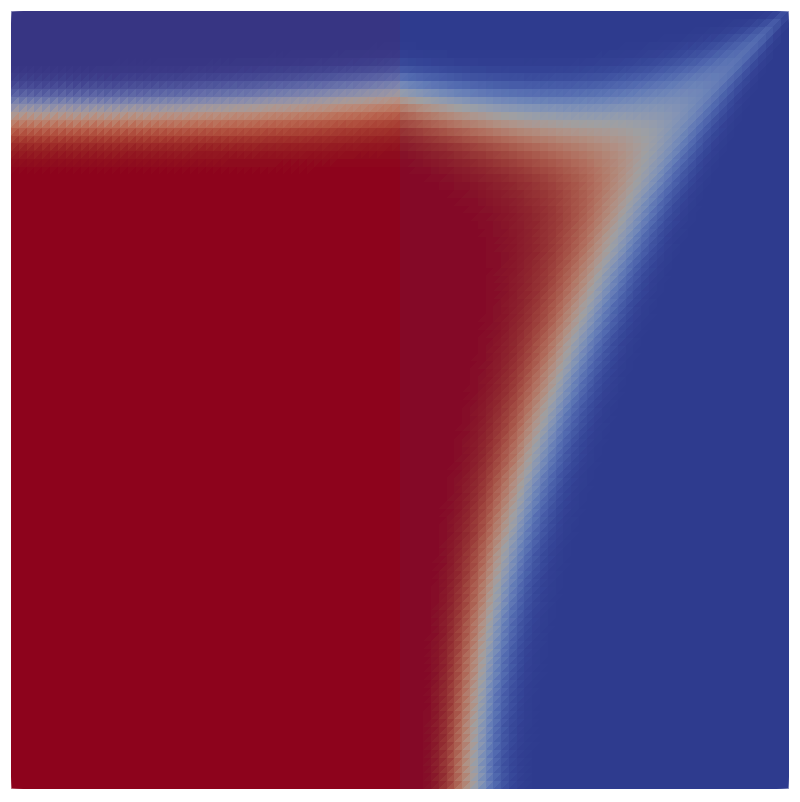}
\includegraphics[width=0.22\textwidth]{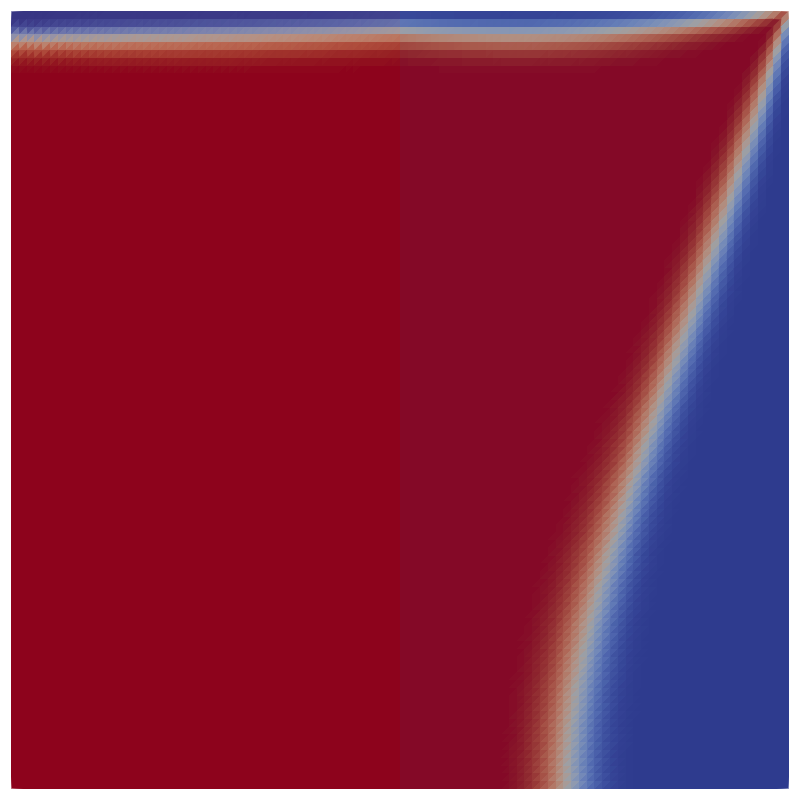}
\includegraphics[width=0.07\textwidth]{fig/range.png}\\
\hspace{-0.5pt}
\subfigure[$T=0$]{\includegraphics[width=0.22\textwidth]{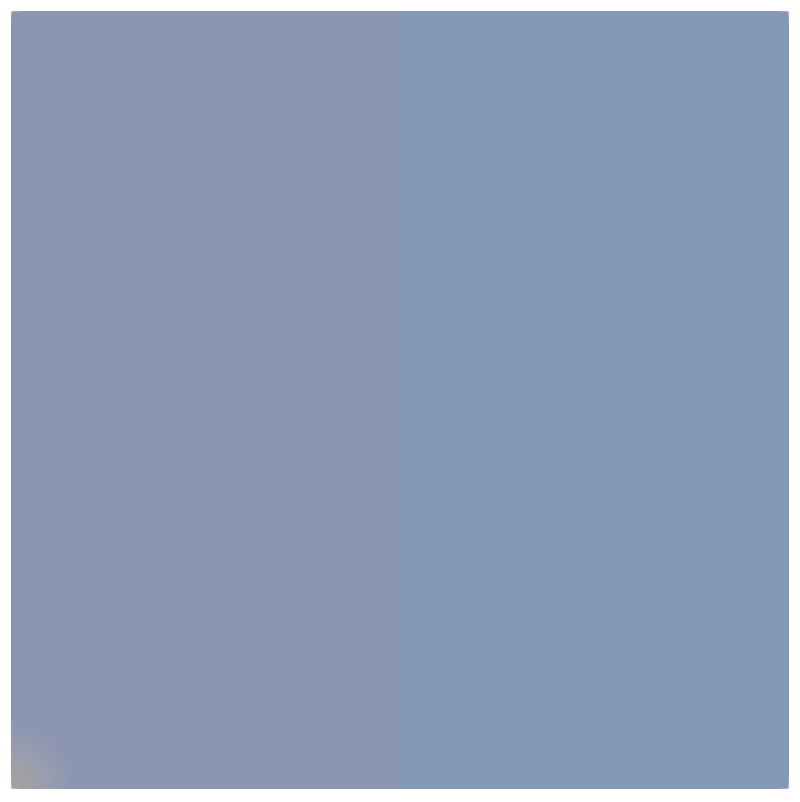}}
\subfigure[$T=4$]{\includegraphics[width=0.22\textwidth]{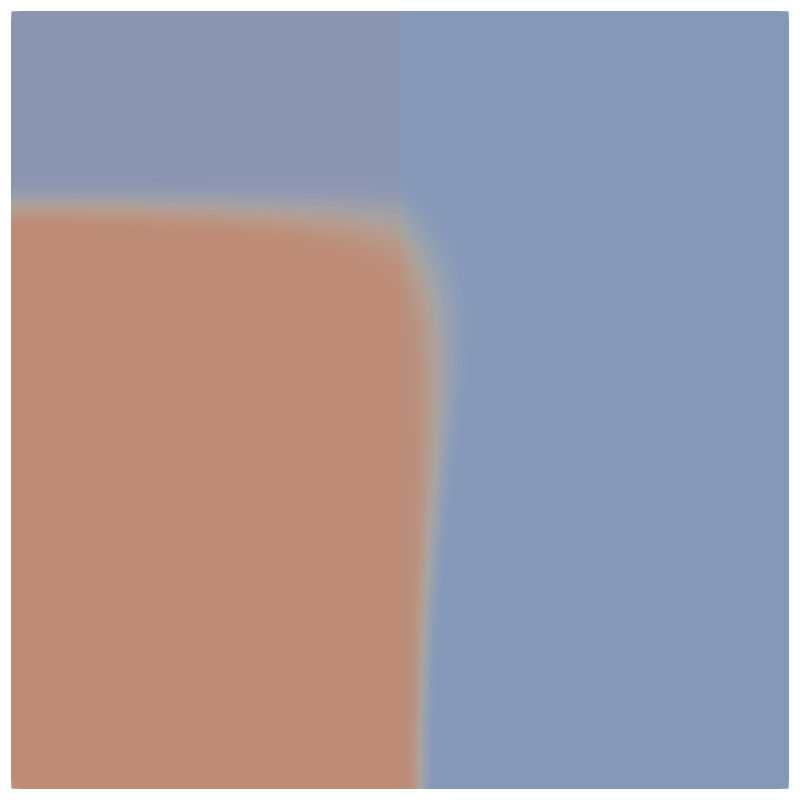}}
\subfigure[$T=6$]{\includegraphics[width=0.22\textwidth]{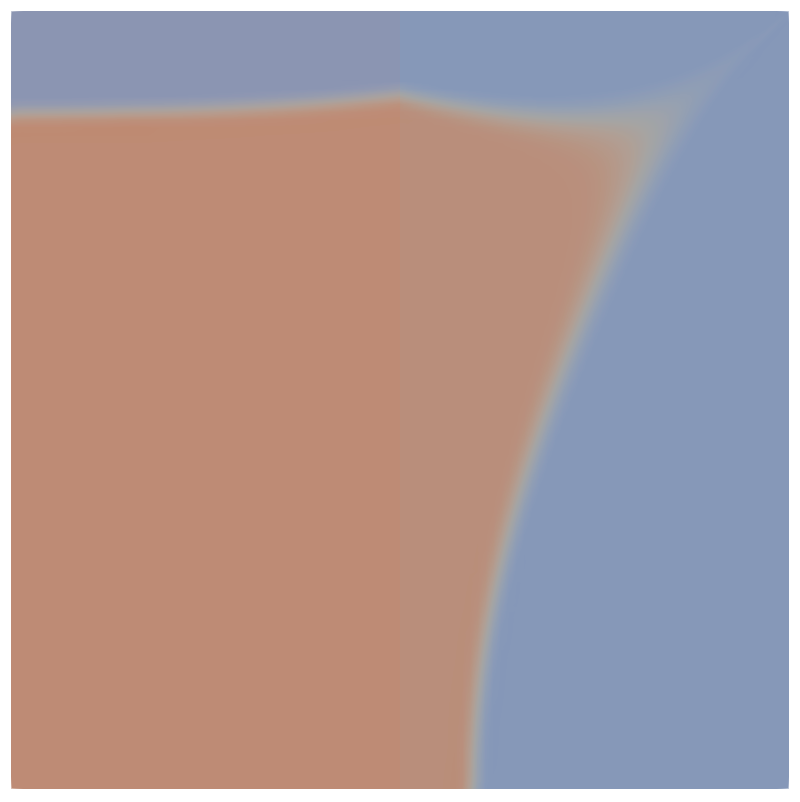}}
\subfigure[$T=10$]{\includegraphics[width=0.22\textwidth]{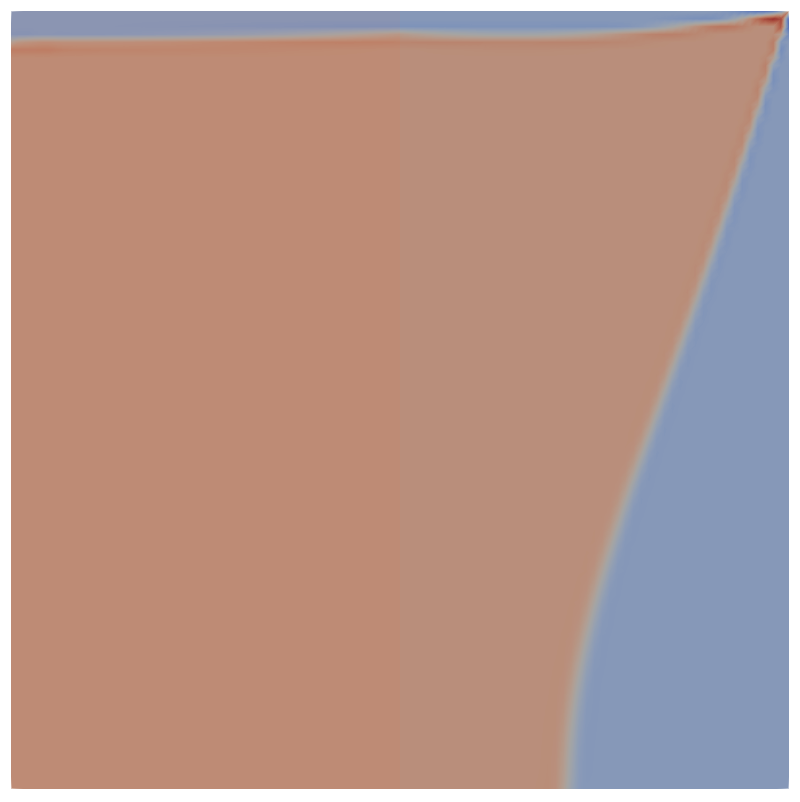}}
\includegraphics[width=0.07\textwidth]{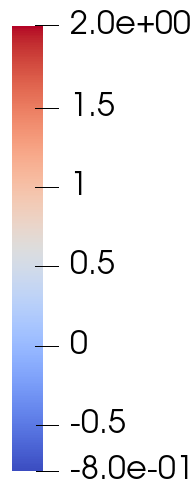}
    \caption{Concentration values at $T=0, 4, 6, 10$ with injection and extraction. The first row used the piecewise constant function and the second row used the piecewise linear function.}
    \label{fig: pointsource}
\end{figure}

\subsection{2D case to study the error accumulation for concentration}
In this section, we investigate the impact of local conservation flux on reproducing the constant solution.
We consider the transport and flow equation on the domain $\Omega=(0,1)^2$. Choose $D$ and $K$ are constant diagonal tensors with $D_{ii}=0$ and $K_{ii}=10.0$. Let  initial condition $c^0=1$ and inflow concentration $c_{\rm I}=1$. The boundary conditions are chosen as follows: $g=100$ on the right boundary, $g=0$ on the bottom boundary, and homogeneous Neumann boundary conditions for the others. The simulation is run until time $T=1$ with $\Delta t=0.01$ and $h=1/64$. For the penalty parameter in the flow equation, we set $\sigma=100/h$. Given that the exact solution in this case is $c=1$, Table \ref{2Derror} presents the error results obtained using different types of DG methods for the flow equation, namely: IIPG, NIPG and SIPG.
\begin{table}[h]\label{2Derror}
\centering
\caption{The $L^2$-error obtained when using piecewise constant and piecewise linear functions to solve the flow equation and various types of DG methods for the transport equation with the exact solution $c=1$.}
\begin{tabular}{@{}cccc}
\hline
& $\theta_f=0$& $\theta_f=1$ &$\theta_f=-1$\\ 
\hline
$DG_0$& 2.3696e-11&  8.1478e-11&2.1508e-11\\
 $DG_1$& 7.2169e-12& 5.1715e-3&8.9678e-3\\
 \hline
\end{tabular}
\end{table}

\section{Conclusions}\label{con} 
In this paper, we introduce the new concept of the locally conservative flux, i.e., the local conservation of degree $k \geq 0$. This can be interpreted as the intermediate local conservation between the standard local conservation and the strong conservation. Under this condition, 
we prove the $L^2$-stability for the transport equation.
Besides, we show how the positivity and maximum principle can be obtained when the piecewise constant function is used for approximating the transport. This resolves the long standing open question. Finally, several numerical experiments are carried out to verify the theoretical results.

\section*{Acknowledgements}
SG is supported in part by National Natural Science Foundation of China (Grant number 12201535), Guangdong Basic and Applied Basic Research Foundation (Grant number 2023A1515011651), and Shenzhen Stability Science Program 2022. YL is supported in part by NSF-DMS 228499. YL is supported in part by NSF-DMS 2110728. YY is supported in part by Hetao Shenzhen-Hong Kong Science and Technology Innovation Cooperation Zone Project (No.HZQSWS-KCCYB-2024016).

\bibliographystyle{siamplain}
\bibliography{xyz,ref}



\end{document}

%% file: ex_shared.tex
\usepackage{lipsum}
\usepackage{amsfonts}
\usepackage{graphicx}
\usepackage{epstopdf}
\usepackage{algorithmic}
\ifpdf
  \DeclareGraphicsExtensions{.eps,.pdf,.png,.jpg}
\else
  \DeclareGraphicsExtensions{.eps}
\fi

\newsiamremark{remark}{Remark}
\newsiamremark{hypothesis}{Hypothesis}
\crefname{hypothesis}{Hypothesis}{Hypotheses}
\newsiamthm{claim}{Claim}

\headers{Positivity and maximum principle preserving DG}{Shihua Gong, Youngju Lee, Yukun Li, Yue Yu}

\title{Positivity and maximum principle preserving discontinuous Galerkin finite element schemes for a coupled flow and transport
}

\author{Shihua Gong\footnotemark[4]
\thanks{School of Science and Engineering, The Chinese University of Hong Kong, Shenzhen, Guangdong 518172, China(\email{gongshihua@cuhk.edu.cn}).}
\and 
Young-Ju Lee\thanks{Department of Mathematics, Texas State University, TX 78666 (\email{yjlee@txstate.edu}).}
\and 
Yukun Li \thanks{Department of Mathematics, University of Central Florida(\email{yukun.li@ucf.edu}).}
\and 
Yue Yu\footnotemark[1]
\thanks{Shenzhen International Center For Industrial And Applied Mathematics, Shenzhen Research Institute of Big Data(\email{223010135@link.cuhk.edu.cn}).}}

\usepackage{amsopn}

\makeatletter
\newcommand*{\addFileDependency}[1]{
  \typeout{(#1)}
  \@addtofilelist{#1}
  \IfFileExists{#1}{}{\typeout{No file #1.}}
}
\makeatother

\newcommand*{\myexternaldocument}[1]{%
    \externaldocument{#1}%
    \addFileDependency{#1.tex}%
    \addFileDependency{#1.aux}%
}